\documentclass[]{interact}

\usepackage{epstopdf}
\usepackage[caption=false]{subfig}

\usepackage[numbers,sort&compress]{natbib}
\bibpunct[, ]{[}{]}{,}{n}{,}{,}
\makeatletter
\def\NAT@def@citea{\def\@citea{\NAT@separator}}
\makeatother
\usepackage{hyperref}
\hypersetup{
    colorlinks=true,
    linkcolor=blue, 
    citecolor=red, 
    urlcolor=blue  } 
\usepackage{mathptmx}   

\theoremstyle{plain}
\newtheorem{theorem}{Theorem}[section]
\newtheorem{lemma}[theorem]{Lemma}
\newtheorem{corollary}[theorem]{Corollary}
\newtheorem{proposition}[theorem]{Proposition}

\theoremstyle{definition}
\newtheorem{definition}[theorem]{Definition}
\newtheorem{example}[theorem]{Example}

\theoremstyle{remark}
\newtheorem{remark}{Remark}

\theoremstyle{plain}
{\bf}{\it}

\theoremstyle{plain}
 {\bf}{\it}

\theoremstyle{plain}
 {\bf}{\it}

\theoremstyle{plain}
 {\bf}{\it}

\theoremstyle{plain}
 {\bf}{\it}

\theoremstyle{plain}
 {\bf}{\it}

\theoremstyle{plain}
 {\bf}{\it}

\theoremstyle{plain}
 {\bf}{\it}

\theoremstyle{plain}
 {\bf}{\it}

\theoremstyle{plain}
 {\bf}{\it}

\theoremstyle{plain}
 {\bf}{\it}

\theoremstyle{plain}
 {\bf}{\it}

\theoremstyle{plain}
 {\bf}{\it}

\theoremstyle{plain}
 {\bf}{\it}

\theoremstyle{plain}
 {\bf}{\it}

\theoremstyle{plain}
 {\bf}{\it}

\theoremstyle{plain}
 {\bf}{\it}

\usepackage{color}
\usepackage[colorinlistoftodos,disable]{todonotes}

\newcommand{\al}{\alpha}
\newcommand{\be}{\beta}

\newcommand{\la}{\lambda}
\newcommand{\de}{\delta}
\newcommand{\eps}{\varepsilon}
\newcommand{\bx}{\bar x}

\newcommand{\iv}{^{-1} }

\newcommand {\R} {\mathbb R}
\newcommand {\N} {\mathbb N}

\newcommand {\B} {\mathbb B}

\newcommand {\dom} {{\rm dom}\,}
\newcommand {\epi} {{\rm epi}\,}

\newcommand {\sd} {\partial}

\newcommand{\folgt}{$ \Rightarrow\ $}

\def\nbh{neighbourhood}
\def\es{\emptyset}

\def\LHS{left-hand side}
\def\RHS{right-hand side}

\def\EVP{Ekeland variational principle}
\def\Fr{Fr\'echet}

\newcommand{\norm}[1]{\left\Vert#1\right\Vert}

\newcommand{\ang}[1]{\left\langle #1 \right\rangle}
\newcommand{\qdtx}[1]{\quad\mbox{#1}\quad}
\newcommand{\AND}{\quad\mbox{and}\quad}
\newcounter{mycount}

\newcommand{\AK}[1]{\todo[inline]{AK {#1}}}


\renewcommand {\theenumi} {\rm(\roman{enumi})}

\begin{document}

\title{Necessary Conditions for Non-Intersection of Collections of Sets}

\author{
\name{Hoa T. Bui\textsuperscript{a,b} and Alexander Y. Kruger\textsuperscript{b}}
\thanks{CONTACT Alexander Y. Kruger. Email: a.kruger@federation.edu.au}
\affil{\textsuperscript{a} School of Electrical Engineering, Computing, and Mathematical Sciences, Curtin University,
Perth, WA, 6102, Australia; \textsuperscript{b} Centre for Informatics and Applied Optimization, Engineering, Information Technology and Physical Sciences,
Federation University, POB 663, Ballarat, Vic, 3350, Australia}
\vspace{5mm}
{Dedicated to Alfredo Iusem on his 70th birthday}
}

\maketitle

\begin{abstract}
This paper continues studies of non-intersection properties of finite collections of sets initiated 40 years ago by the \emph{extremal principle}.
We study elementary \emph{non-inter\-section} properties of collections of sets, making the core of the conventional definitions of \emph{extremality} and \emph{stationarity}.
In the setting of general Banach/Asplund spaces, we establish new primal (slope) and dual (generalized separation) necessary conditions for these non-inter\-section properties.
The results are applied to convergence analysis of alternating projections.
\end{abstract}


\begin{amscode}
49J52; 49J53; 49K40; 90C30
\end{amscode}


\section{Introduction}\label{Int}

This paper continues studies of geometric non-intersection properties of finite collections of sets initiated by the \emph{extremal principle} \cite{KruMor79,KruMor80.2,KruMor80}.
Models involving collections of sets have proved their usefulness in analysis and optimization, with non-intersection properties (or their absence) being at the core of many applications: recall the ubiquitous convex \emph{separation theorem, Dubovitskii--Milyutin formalism} \cite{DubMil65} and various \emph{transversality/regularity} properties \cite{Kru05,Kru06,Kru09,DruIofLew15,Iof17,Iof17.2, KruLukTha17,Kru18,KruLukTha18,BivKraRib20,ThaBuiCuoVer20}.

The classical separation theorem states that two convex sets
such that
one of the sets does not meet the nonempty interior of the other set, can be separated by a hyperplane determined by a nonzero dual space vector.
Similarly, the extremal principle provides a dual space \emph{generalized separation} characterization of a certain extremal property of a pair of sets without assuming any set to be convex or have nonempty interior.
This extremal property (\emph{extremality}) provides a very general model that embraces many optimality notions.
Thus, the extremal principle can substitute the conventional separation theorem in the nonconvex settings, e.g., when proving necessary optimality conditions, subdifferential, normal cone and coderivative calculus rules \cite{KruMor79,KruMor80.2,KruMor80,Kru81.2,Kru85.1,BorJof98, Iof98,Kru03, BorZhu05, Mor06.1}.

The extremality assumption in the conventional extremal principle was successively relaxed to \emph{local extremality} \cite{Kru81.2}, \emph{stationarity} and \emph{approximate stationarity} \cite{Kru98,Kru02,Kru04,Kru09}, while preserving the generalized separation conclusion, and without significant changes in the original proof.
We refer the readers to \cite[Section~2.6]{Mor06.1} and \cite{BuiKru18} for historical comments.

Below we recall the conventional definitions of the extremality/stationarity properties of a collection of $n\ge2$ arbitrary subsets $\Omega_1,\ldots,\Omega_n$ of a normed vector space, having a common point.
We write $\{\Omega_1,\ldots,\Omega_n\}$ to denote the collection of sets as a single object.

\begin{definition}\label{D8}
Let $\Omega_1,\ldots,\Omega_n$ be subsets of a normed vector space $X$ and ${\bx\in\cap_{i=1}^n\Omega_i}$.
The collection $\{\Omega_1,\ldots,\Omega_n\}$ is
\begin{enumerate}
\item
\textbf{extremal} if
for any $\eps>0$, there exist vectors
$a_i\in{X}$ $(i=1,\ldots,n)$ satisfying
\begin{gather}\label{D1-1}
\bigcap_{i=1}^n(\Omega_i-a_i)=\emptyset,
\\\label{D1-2}
\max_{1\le i\le n}\norm{a_i}<\eps;
\end{gather}

\item
\textbf{locally extremal} at $\bx$ if there exists a number $\rho\in]0,+\infty]$ such that,
for any $\eps>0$, there are vectors
$a_i\in{X}$ $(i=1,\ldots,n)$ satisfying \eqref{D1-2} and
\begin{gather}\label{D1-3}
\bigcap_{i=1}^n(\Omega_i-a_i)\cap{B}_\rho(\bar{x})
=\emptyset;
\end{gather}

\item
\textbf{stationary} at $\bx$ if
for any $\eps>0$,
there exist a number $\rho\in]0,\eps[$
and vectors
$a_i\in{X}$ $(i=1,\ldots,n)$ satisfying \eqref{D1-3} and
\begin{gather}\label{D1-4}
\max_{1\le i\le n}\norm{a_i}<\eps\rho;
\end{gather}

\item
\textbf{approximately stationary} at $\bx$ if
for any $\eps>0$,
there exist a number $\rho\in]0,\eps[$, points $\omega_i\in\Omega_i\cap \B_\eps(\bx)$
and vectors
$a_i\in{X}$ $(i=1,\ldots,n)$ satisfying \eqref{D1-4} and
\begin{gather}\label{D1-5}
\bigcap_{i=1}^n(\Omega_i-\omega_i-a_i)\cap(\rho\B)
=\emptyset.
\end{gather}
\end{enumerate}
\end{definition}

The relationships between the properties in Definition~\ref{D8} are straightforward:
{\rm (i)~$\Rightarrow$~(ii) \folgt (iii) \folgt (iv)}.
If the sets are convex, then all four properties are equivalent \cite[Proposition~14]{Kru05}.
Note that the definition of local extremality in item (ii) allows for the value $\rho=+\infty$ and, thus, covers the (global) extremality property in item (i).

The approximate stationarity, being the weakest of the four mentioned primal space properties, is in fact equivalent to the dual generalized separation in the conclusion of the conventional extremal principle in Asplund spaces.
This result is known as the \emph{extended extremal principle} \cite{Kru98,Kru02,Kru03}.

\begin{theorem}[Extended extremal principle]\label{T11}
Let $\Omega_1,\ldots,\Omega_n$ be closed subsets of an Asplund space $X$ and
${\bx\in\cap_{i=1}^n\Omega_i}$.
The following conditions are equivalent:
\sloppy
\begin{enumerate}
\item
the collection $\{\Omega_1,\ldots,\Omega_n\}$ is approximately stationary at $\bx$;
\item
for any $\eps>0$, there exist points
$\omega_i\in\Omega_i\cap{B}_\eps(\bx)$ and vectors $x^*_i\in N^F_{\Omega_i}(\omega_i)$ ${(i=1,\ldots,n)}$ such that
$\norm{\sum_{i=1}^n x^*_i}<\eps$ and
$\sum_{i=1}^{n}\norm{x_i^*}=1$;

\item
for any $\eps>0$, there exist points
$\omega_i\in\Omega_i\cap{B}_\eps(\bx)$ and vectors $x_i^*\in X^*$ ${(i=1,\ldots,n)}$ such that
$\sum_{i=1}^nd(x^*_i,N^F_{\Omega_i}(\omega_i))<\eps$,
$\sum_{i=1}^n x^*_i=0$ and
$\sum_{i=1}^{n}\norm{x_i^*}=1$.
\end{enumerate}
\end{theorem}

The equivalent conditions (ii) and (iii) in Theorem~\ref{T11} have been used interchangeably (together
with several their modifications) since 1979 in the concluding part of the extremal principle and its extensions
\cite{KruMor79,KruMor80.2,KruMor80,Kru81.2,Kru85.1,MorSha96, Kru98,Kru00,Kru03,BorZhu05,Mor06.1}.
The necessity of conditions (ii) and (iii) for the approximate stationarity can be easily extended to general Banach spaces if \Fr\ normal cones are replaced by Clarke or $G$-normal cones; cf. \cite[Remark 2.1(iii)]{BuiKru19}.

The proof of the conventional extremal principle and all its subsequent extensions is based on the two fundamental results of variational analysis:
\begin{itemize}
\item
Ekeland variational principle,
\item
a sum rule for the appropriate subdifferential.
\end{itemize}

The exact opposite of the approximate stationarity happens to be another important property, currently called \emph{transversality}.
Along with other (weaker) transversality properties, it is frequently used in constraint qualifications, qualification conditions in subdifferential, normal cone and coderivative calculus, and convergence analysis of computational algorithms \cite{BauBor93,Mor06.1,LewLukMal09,DruIofLew15,KruLukTha18}.
In particular, under the transversality assumption on the sets, the alternating projections converge linearly to a point in the intersection.
The extended extremal principle automatically provides an equivalent dual characterization of transversality.
In its turn, transversality is closely related (in a sense equivalent) to the fundamental property of \emph{metric regularity} of set-valued mappings.
Many primal and dual characterizations of transversality properties have been established recently \cite{KruTha15,KruTha16,KruLukTha17,Kru18,KruLukTha18, BuiCuoKru20, ThaBuiCuoVer20}, mostly in the linear setting.
Studies of nonlinear versions of these properties have only started \cite{CuoKru,CuoKru20,CuoKru20.2,CuoKru6}.

The properties in Definition~\ref{D8} involve translations of all the sets determined by vectors $a_i\in{X}$ $(i=1,\ldots,n)$.
In all the properties it is sufficient to consider translations of all but one sets.
This simple observation leads to asymmetric conditions which can be useful, especially in the case $n=2$.
The conditions in the next statement correspond to setting $a_n:=0$ in the corresponding conditions in Definition~\ref{D8}.

\begin{proposition}\label{P10}
Let $\Omega_1,\ldots,\Omega_n$ be subsets of a normed vector space $X$ and ${\bx\in\cap_{i=1}^n\Omega_i}$.
The collection $\{\Omega_1,\ldots,\Omega_n\}$ is
\begin{enumerate}
\item
extremal if and only if,
for any $\eps>0$, there exist vectors
$a_i\in{X}$ $(i=1,\ldots,n-1)$ satisfying
\begin{gather}\label{P10-1}
\bigcap_{i=1}^{n-1}(\Omega_i-a_i)\cap\Omega_n=\emptyset,
\\\label{P10-2}
\max_{1\le i\le n-1}\norm{a_i}<\eps;
\end{gather}

\item
locally extremal at $\bx$ if and only if there exists a number $\rho\in]0,+\infty]$ such that,
for any $\eps>0$, there exist vectors
$a_i\in{X}$ $(i=1,\ldots,n-1)$ satisfying \eqref{P10-2} and
\begin{gather}\label{P10-3}
\bigcap_{i=1}^{n-1}(\Omega_i-a_i) \cap\Omega_n\cap{B}_\rho(\bar{x})=\emptyset;
\end{gather}
moreover, if $\{\Omega_1,\ldots,\Omega_n\}$ is locally extremal at $\bx$ with some $\rho\in]0,+\infty]$, then the above condition holds with any $\rho'\in]0,\rho[$ in place of $\rho$;
if $\rho=+\infty$, one can take $\rho':=+\infty$;

\item
stationary at $\bx$ if and only if,
for any $\eps>0$,
there exist a number $\rho\in]0,\eps[$
and vectors
$a_i\in{X}$ $(i=1,\ldots,n-1)$ satisfying \eqref{P10-3} and
\begin{gather}\label{P10-4}
\max_{1\le i\le n-1}\norm{a_i}<\eps\rho;
\end{gather}

\item
approximately stationary at $\bx$ if and only if,
for any $\eps>0$,
there exist a number $\rho\in]0,\eps[$, points $\omega_i\in\Omega_i\cap \B_\eps(\bx)$ $(i=1,\ldots,n)$
and vectors
$a_i\in{X}$ ${(i=1,\ldots,n-1)}$ satisfying \eqref{P10-4} and
\begin{gather}\label{P10-5}
\bigcap_{i=1}^{n-1}(\Omega_i-\omega_i-a_i) \cap(\Omega_n-\omega_n) \cap(\rho\B)
=\emptyset.
\end{gather}
\end{enumerate}
\end{proposition}

\if{
\begin{proof}
The properties above imply the corresponding ones in Definition~\ref{D8} with $a_n=0$.
For the opposite implication, given vectors $a_i\in{X}$ $(i=1,\ldots,n)$, it is natural to consider vectors $a'_i:=a_i-a_n$ $(i=1,\ldots,n-1)$.
Then
\begin{gather}\label{P10P1}
\bigcap_{i=1}^{n-1}(\Omega_i-a'_i)\cap\Omega_n =\bigcap_{i=1}^{n}(\Omega_i-a_i)+a_n.
\end{gather}
Hence, condition \eqref{D1-1} implies condition \eqref{P10-1} with the collection of vectors $a'_i$'s in place of $a_i$'s.
Given an $\eps>0$, if one of the conditions \eqref{D1-2} or \eqref{D1-4} is satisfied with some $\eps'\in]0,\eps/2]$ in place of $\eps$, then the corresponding condition \eqref{P10-2} or \eqref{P10-4} is satisfied with the collection $a'_i$'s in place of $a_i$'s.
Given a $\rho\in]0,+\infty]$ and a $\rho'\in]0,\rho[$, one can take a smaller $\eps'>0$ to ensure that $\rho'+\eps'<\rho$.
Then, in view of \eqref{P10P1} and assuming \eqref{D1-2} with $\eps'$ in place of $\eps$, we have
\begin{gather*}
\bigcap_{i=1}^{n-1}(\Omega_i-a'_i) \cap\Omega_n\cap{B}_{\rho'}(\bar{x})\subset \bigcap_{i=1}^{n}(\Omega_i-a_i)\cap{B}_{\rho}(\bar{x}).
\end{gather*}
The above inclusion obviously holds also with $\rho'=\rho=+\infty$.
Thus, condition \eqref{D1-3} implies condition \eqref{P10-3} with $a'_i$'s and $\rho'$ in place of $a_i$'s and $\rho$, respectively.
Observe that conditions \eqref{D1-5} and \eqref{P10-5} are actually conditions \eqref{D1-3} and \eqref{P10-3}, respectively, with the sets $\Omega_i-\omega_i$ ${(i=1,\ldots,n)}$ in place of $\Omega_i$ $(i=1,\ldots,n)$.
Hence, condition \eqref{D1-5} implies condition \eqref{P10-5} with $a'_i$'s and $\rho'$ in place of $a_i$'s and $\rho$, respectively.
It follows that
the properties in Definition~\ref{D8} imply the corresponding ones in the proposition with $a'_i$'s in place of $a_i$'s, as well as the `moreover' part in item (ii).
\end{proof}
}\fi

The definitions of the extremality/stationarity properties as well as their equivalent asymmetric characterizations in Proposition~\ref{P10} involve non-intersection conditions \eqref{D1-1}, \eqref{D1-3}, \eqref{D1-5}, \eqref{P10-1}, \eqref{P10-3} and \eqref{P10-5} for certain `small' (in the sense of \eqref{D1-2}, \eqref{D1-4}, \eqref{P10-2} or \eqref{P10-4}) translations determined by $n$ or $n-1$ vectors $a_i\in{X}$ of either the original sets $\Omega_1,\ldots,\Omega_n$ in the case of \eqref{D1-1}, \eqref{D1-3}, \eqref{P10-1} and \eqref{P10-3}, or the sets $\Omega_1-\omega_1,\ldots,\Omega_n-\omega_n$ in the case of \eqref{D1-5} and \eqref{P10-5}.

The mentioned six non-intersection conditions are in fact variations of each other, and can all be considered as variations of a single condition, e.g., condition \eqref{P10-1}.
\if{
Indeed, conditions \eqref{P10-1}, \eqref{P10-3} and \eqref{P10-5} are particular cases of conditions \eqref{D1-1}, \eqref{D1-3} and \eqref{D1-5}, respectively, with $a_n=0$.
Conditions \eqref{D1-1} and \eqref{P10-1} are particular cases of conditions \eqref{D1-3} and \eqref{P10-3}, respectively, with $\rho=+\infty$, while conditions \eqref{D1-3} and \eqref{P10-3}, in their turn, are particular cases of conditions \eqref{D1-5} and \eqref{P10-5}, respectively, with $\omega_1=\ldots=\omega_n=\bx$.
Moreover, condition \eqref{P10-3} is a particular case of condition \eqref{P10-1} with $\Omega_n':=\Omega_n\cap{B}_\rho(\bar{x})$ in place of $\Omega_n$, while conditions \eqref{D1-5} and \eqref{P10-5} are particular cases of conditions \eqref{D1-3} and \eqref{P10-3}, respectively, with the sets $\Omega_1-\omega_1,\ldots,\Omega_n-\omega_n$ in place of $\Omega_1,\ldots,\Omega_n$ and $0$ in place of $\bx$.
}\fi
Conditions \eqref{P10-1}, \eqref{P10-3} and \eqref{P10-5} are obviously particular cases of conditions \eqref{D1-1}, \eqref{D1-3} and \eqref{D1-5}, respectively, with $a_n=0$.
The first three conditions are in a sense equivalent.
Indeed, condition \eqref{P10-1} is a particular case of condition \eqref{P10-3} with $\rho=+\infty$, while condition \eqref{P10-3}, in its turn, is a particular case of condition \eqref{P10-5} with $\omega_1=\ldots=\omega_n=\bx$.
Conversely, condition \eqref{P10-3} is a particular case of condition \eqref{P10-1} with $\Omega_n':=\Omega_n\cap{B}_\rho(\bar{x})$ in place of $\Omega_n$, while condition \eqref{P10-5} is a particular case of condition \eqref{P10-3} with the sets $\Omega_1-\omega_1,\ldots,\Omega_n-\omega_n$ in place of $\Omega_1,\ldots,\Omega_n$ and $0$ in place of $\bx$.
Now notice that the seemingly more general conditions \eqref{D1-1}, \eqref{D1-3} and \eqref{D1-5} reduce to a condition of the type \eqref{P10-1} for a collection of $n+1$ sets.
In the case of conditions \eqref{D1-1} and \eqref{D1-3}, it suffices to add to $\Omega_1,\ldots,\Omega_n$ the sets $X$ and $B_\rho(\bx)$, respectively, while condition \eqref{D1-5} is a particular case of condition \eqref{D1-3} with the sets $\Omega_1-\omega_1,\ldots,\Omega_n-\omega_n$ in place of $\Omega_1,\ldots,\Omega_n$ and $0$ in place of $\bx$.

Choosing, e.g., condition \eqref{P10-1}, the simplest of the three equivalent non-intersection conditions, together with the accompanying restriction \eqref{P10-2} on the size of translations, and introducing a convex continuous function
\begin{gather}\label{f0}
f(u_1,\ldots,u_{n}):=\max\limits_{1\le{i}\le{n}-1} \norm{u_i-a_i-u_n},
\quad
u_1,\ldots,u_n\in X,
\end{gather}
on $X^n$, one immediately obtains the estimates:
\begin{gather}\label{12}
f(u_1,\ldots,u_n)>0 \qdtx{for all} u_i\in\Omega_i \;(i=1,\ldots,n),
\\\label{13}
f(\bx,\ldots,\bx)=\max_{1\le i\le n-1}\norm{a_i} <\eps,
\end{gather}
which obviously entice one to apply the Ekeland variational principle.
Moreover, nice properties of the function \eqref{f0} together with condition \eqref{12} allow one to apply then an appropriate
subdifferential sum rule to establish dual conditions in terms of normals to the sets $\Omega_1,\ldots,\Omega_n$.
Arguments of this kind make the core part of the original (infinite dimensional) proof of the extremal principle and all its subsequent extensions including Theorem~\ref{T11}.

Note that conditions \eqref{P10-1} and \eqref{P10-2} and the function \eqref{f0} leading to the above arguments are formulated for fixed number $\eps>0$ and vectors $a_i\in X$, while all the properties in Definition~\ref{D8}, Proposition~\ref{P10} and Theorem~\ref{T11} (as well as in the conventional extremal principle) are formulated `for any $\eps>0$'.
This observation shows that the conventional (extended) extremal principle, although sufficient for most applications, does not use the full potential of the above arguments.

Recently, a few situations have been identified
where the conventional (extended) extremal principle fails, while the above arguments are still applicable.
They were successfully used, for instance, in \cite{KruLop12.1} to extend the extremal principle to infinite collections of sets.
For that purpose, a more universal `fixed $\eps$' type statement \cite[Theorem~3.1]{KruLop12.1} was formulated, with its proof encapsulating the above arguments, from which both the conventional extremal principle and its extension to infinite collections followed as immediate corollaries.
Another two lemmas of this kind were established by Zheng and Ng in \cite{ZheNg06}.
They have been further refined and strengthened in \cite{ZheNg11,ZheYanZou17} producing a `unified separation theorem', used in \cite{ZheNg11} to prove fuzzy multiplier rules in set-valued optimization problems.
The `fixed $\eps$' type statements from \cite{ZheNg06, ZheNg11,ZheYanZou17,KruLop12.1} have been compared and further unified in \cite{BuiKru18,BuiKru19}.
Such statements
open a way for characterizing other extremality-like properties.

The function \eqref{f0} playing the key role in the above arguments is a composition of a linear mapping from $X^n$ to $X^{n-1}$ and the $l_{\infty}$ (maximum) norm on $\R^{n-1}$.
It is easy to see that any norm on $\R^{n-1}$ can be used as
the outer function ensuring the same estimates \eqref{12} and \eqref{13}; cf., e.g., \cite{ZheNg11}.
This would of course lead to appropriate straightforward changes in the dual conditions, e.g., the sum of norms in parts (ii) and (iii) of  Theorem~\ref{T11} would have to be replaced by the corresponding finite dimensional dual norm.

\if{
To quantify non-intersection properties of a collection of sets, one can use the following distance-like quantity (\emph{nonintersect index} \cite{ZheNg11}):
\begin{gather}\label{d}
d(\Omega_1,\ldots,\Omega_{n}) :=\inf_{u_i\in\Omega_i\;(i=1,\ldots,n)}\max_{1\le{i}\le{n}-1} \norm{u_n-u_i},
\end{gather}
or some other quantity of this kind; cf. \cite{BuiKru19}.
The maximum in \eqref{d} obviously corresponds to the $l_{\infty}$ (uniform) norm on $\R^{n-1}$.
It can be replaced by any other norm on $\R^{n-1}$; cf. \cite{ZheNg11}.
If $\cap_{i=1}^n\Omega_i\ne\es$, then
$d(\Omega_1-a_1,\ldots,\Omega_{n-1}-a_{n-1},\Omega_{n}) \le\max_{1\le i\le n-1}\norm{a_i}$.
}\fi

Now the reader should be ready to observe that in the above arguments, leading to the application of the \EVP, the function $f$ does not have to be of the form \eqref{f0} (with any finite-dimensional norm as the outer function).
It only needs to satisfy conditions \eqref{12} and \eqref{13}.
Of course, if one targets dual conditions, the function needs to be subdifferentiable (in some sense, with the subdifferential mapping possessing a natural sum rule).
This more general approach adds flexibility and expands applicability of the results.
It has been successfully used recently in \cite{CuoKru,CuoKru20,CuoKru20.2,CuoKru6} when studying transversality properties.
In this paper, we are applying it to establishing new `nonlinear' primal and dual necessary conditions for the elementary non-intersection properties in Definition~\ref{D8}.

The function $f$ can be rather general.
It does not even have to be continuous,
but we do not go that far in this paper.
We consider compositions of the conventional distance-type functions and continuous strictly increasing functions from $\R_+$ to $\R_+$.
The resulting estimates are further specified for the cases when the outer function is (continuously) differentiable, and when it is a power function (H\"older estimates).

The nonlinear estimates established in this paper can be automatically used to replace the elementary `fixed $\eps$' type statements from \cite{ZheNg06, ZheNg11,ZheYanZou17,KruLop12.1,BuiKru18,BuiKru19} when formulating necessary conditions for the extremality/stationarity of finite and infinite collections of sets and the corresponding optimality conditions like, e.g., those in \cite{KruLop12.2}, adding another degree of freedom.
We leave such illustrations for future publications.
Instead, we provide an application to convergence analysis of alternating projections, extending the $R-$linear convergence estimates from~\cite{DruIofLew15}.
We show that under certain weaker nonlinear conditions, the alternating projections still converge, though with a different rate.
The nonconsistent case is explored and finite convergence of the alternating projections is established.

The structure of the paper is as follows.
Section~\ref{Pre} recalls some definitions and facts used throughout the paper, particularly,
the Ekeland variational principle and slope chain rule which are the core tools for proving necessary conditions, and three types of subdifferential sum rules used when deducing dual conditions.

In Section~\ref{Slope}, we establish slope necessary conditions for the elementary non-intersection properties of collections of sets contained in all parts of Definition~\ref{D8} and Proposition~\ref{P10}.
The conditions for the property \eqref{P10-1} in Theorem~\ref{T12} are immediate consequences of the \EVP\ and slope chain rule, with the other conditions being consequences of Theorem~\ref{T12}.
This type of conditions are rather common, e.g., in the closely related area of error bounds \cite{Aze03,NgaThe08,Kru15,AzeCor17,Iof17}, but with regards to the extremality/stationarity/transversality properties of collections of sets, their importance has been underestimated, with the estimates hidden in numerous proofs of dual conditions.
We believe that such conditions can be of importance by themselves, particularly because they are usually stronger and sometimes easier to check than the corresponding dual ones.
Besides, formulating such estimates as separate statements makes the proofs of the dual conditions simpler.

Section~\ref{dual} is devoted to dual necessary conditions.
The main conditions in Theorem~\ref{T17} are consequences of the primal ones in Theorem~\ref{T12} and the subdifferential sum rules, with the other dual conditions being immediate consequences of Theorem~\ref{T17}.
Among other things, nonlinear extensions of \cite[Theorem~3.1]{KruLop12.1} and \cite[Theorems~3.1 and 3.4]{ZheNg11} are obtained.
An additional condition relating primal and dual vectors is added to all dual statements.
The importance of such conditions (coming from the application of a subdifferential sum rule) was first observed and justified by Zheng and Ng \cite{ZheNg11} when formulating their unified separation theorem.
A similar condition was used in \cite{ChuKruYao11} when characterizing calmness of efficient solution maps in parametric vector optimization.

In Section~\ref{AP}, we consider an application of the dual necessary conditions in Section~\ref{dual} to convergence analysis of alternating projections.

\section{Preliminaries}\label{Pre}

Our basic notation is standard, see, e.g., \cite{RocWet98,Mor06.1,DonRoc14}.
Throughout the paper, $X$ is either a metric or, more often, a normed vector space.
The open unit ball in any space is denoted by $\B$, while $\B_\delta(x)$ and $\overline{\B}_\delta(x)$ stand, respectively, for the open and closed balls with center $x$ and radius $\delta>0$.
If not explicitly stated otherwise, products of normed vector spaces are assumed to be equipped with the maximum norm ${\|(x,y)\|:=\max\{\|x\|,\|y\|\}}$, ${(x,y)\in X\times Y}$.
$\R$ and $\R_+$ denote the real line (with the usual norm) and the set of all nonnegative real numbers.
The distance from a point $x$ to a set $\Omega$ is defined by $d(x,\Omega):=\inf_{u \in \Omega}\|u-x\|$, and we use the convention $d(x,\emptyset) = +\infty$.
For an extended-real-valued function $f:X\to\R\cup\{+\infty\}$ on a normed vector space $X$,
its domain is defined
by
$\dom f:=\{x \in X\mid f(x) < +\infty\}$

The key tool in the proof of the main slope necessary conditions is
the celebrated Ekeland variational principle; cf., e.g., \cite{Kru03,Mor06.1,DonRoc14}.

\begin{lemma}[Ekeland variational principle]\label{evp}
Suppose $X$ is a complete metric space, $f: X\to\R\cup\{+\infty\}$ is lower semicontinuous, $\bx\in X$, $\varepsilon>0$ and
$
f(\bar{x})<\inf_{X}f+\varepsilon.
$
Then, for any $\la>0$, there exists an $\hat{x}\in X$ such that
\begin{enumerate}
\item $d(\hat{x},\bx)<\lambda$;
\item $f(\hat{x})\le f(\bx)$;
\item
$f(x)+(\varepsilon/\lambda)d(x,\hat{x})>f(\hat{x})$ for all $x\in X\setminus\{\bx\}.$
\end{enumerate}
\end{lemma}


Given a function $\psi:X\rightarrow \R\cup\{+\infty\}$ on a metric space, its
\emph{slope} \cite{DegMarTos80} (cf. \cite{Aze03,Iof17}) and \emph{nonlocal slope} \cite{Kru15} (cf. \cite{NgaThe08}) at any $x\in\dom\psi$ are defined, respectively by
\begin{align*}
|\nabla \psi|(x):=\limsup_{u\rightarrow x,\,u\ne x} \dfrac{[\psi(x)-\psi(u)]_+}{d(x,u)}
\AND
|\nabla \psi|^\diamond(x):=\sup_{u\ne x} \dfrac{[\psi(x)-\psi(u)_+]_+}{d(x,u)}.
\end{align*}
where $\al_+:=\max\{0,\al\}$.
When $x\notin \dom \psi$, we set $|\nabla\psi|(x):=|\nabla\psi|^\diamond(x):=+\infty$.
Obviously, $0\le|\nabla \psi|(x)\le|\nabla\psi|^\diamond(x)$ for all $x\in X$ with $\psi(x)\ge0$, and both quantities can be infinite.

The next lemma from \cite{CuoKru} provides a chain rule for slopes.
It slightly improves \cite[Lemma~4.1]{AzeCor17}, where $\psi$ and $\varphi$ were assumed lower semicontinuous and continuously differentiable, respectively.
The composition $\varphi\circ\psi$ of a function
$\psi:X\rightarrow\R\cup\{+\infty\}$ on a metric space and a function
$\varphi:\R\rightarrow\R\cup\{+\infty\}$ is understood in the usual sense with the natural convention that $(\varphi\circ\psi)(x)=+\infty$ if $\psi(x)=+\infty$.

\begin{lemma}[Slope chain rule]\label{L2}
Let $X$ be a metric space, $\psi:X\rightarrow\R\cup\{+\infty\}$,
${\varphi:\R\rightarrow\R\cup\{+\infty\}}$,
$x\in\dom\psi$ and $\psi(x)\in\dom\varphi$.
Suppose $\varphi$ is nondecreasing on $\R$ and differentiable at $\psi(x)$, and either
$\varphi'(\psi(x))>0$ or $|\nabla\psi|(x)<+\infty$.
Then
\sloppy
\begin{align*}
|\nabla(\varphi\circ\psi)|(x)=\varphi'(\psi(x))|\nabla\psi|(x).
\end{align*}
\end{lemma}

\begin{remark}
The chain rule in Lemma \ref{L2} is a local result.
Instead of assuming that $\varphi$ is defined on the whole real line, one can assume that $\varphi$ is defined and finite on a closed interval $[\al,\be]$ around the point $\psi(x)$: ${\al<\psi(x)<\be}$.
It is sufficient to redefine the composition $\varphi\circ\psi$ for $x$ with $\psi(x)\notin[\al,\be]$ as follows:
$(\varphi{\circ}\psi)(x):=\varphi(\al)$ if $\psi(x)<\al$, and
$(\varphi{\circ}\psi)(x):=\varphi(\be)$ if $\psi(x)>\be$.
This does not affect the conclusion of the lemma.
\end{remark}


Dual necessary conditions for extremality/stationarity properties require dual tools -- normal cones and subdifferentials.
In this paper, we use Clarke \cite{Cla83} and \Fr\ \cite{Kru03} ones for characterizations in general Banach and Asplund spaces, respectively.

Given a subset $\Omega$ of a normed vector space $X$ and a point $\bx\in \Omega$, the sets
\begin{gather}\label{NC}
N_{\Omega}^F(\bx):= \left\{x^\ast\in X^\ast\mid
\limsup_{\Omega\ni x\to\bar x} \frac {\langle x^\ast,x-\bx\rangle}
{\|x-\bx\|} \le 0 \right\},
\\\label{NCC}
N_{\Omega}^C(\bx):= \left\{x^\ast\in X^\ast\mid
\ang{x^\ast,z}\le0
\qdtx{for all}
z\in T_{\Omega}^C(\bx)\right\}
\end{gather}
are the \emph{Fr\'echet} and \emph{Clarke normal cones} to $\Omega$ at $\bx$.
In the last definition, $T_{\Omega}^C(\bx)$ stands for the \emph{Clarke tangent cone} \cite{Cla83} to $\Omega$ at $\bx$.
The sets \eqref{NC} and \eqref{NCC} are nonempty
closed convex cones satisfying $N_{\Omega}^F(\bx)\subset N_{\Omega}^C(\bx)$.
If $\Omega$ is a convex set, they reduce to the normal cone in the sense of convex analysis:
\begin{gather*}\label{CNC}
N_{\Omega}(\bx):= \left\{x^*\in X^*\mid \langle x^*,x-\bx \rangle \leq 0 \qdtx{for all} x\in \Omega\right\}.
\end{gather*}

Given a function $f:X\to\R\cup\{+\infty\}$ and a point $\bar x\in X$ with $f(\bx)<+\infty$, the \emph{Fr\'echet} and \emph{Clarke subdifferentials} of $f$ at $\bar x$ are defined as \begin{gather}\label{sdF}
\partial^F f(\bar x):=\left\{x^*\in X^*\mid \liminf_{\substack{x\to \bar x}} \dfrac{f(x)-f(\bar x)-\langle x^*,x-\bar x\rangle}{\|x-\bar x\|}\ge 0\right\},
\\\label{sdC}
\partial^Cf(\bx):=\left\{x^*\in X^*\mid \langle x^*,z\rangle \le f^\circ(\bx,z)
\quad\text{for all}\quad
z\in X\right\},
\end{gather}
where $f^\circ(\bx,z)$ is the \emph{Clarke--Rockafellar directional derivative} \cite{Roc79}
of $f$ at $\bx$ in the direction $z\in X$.
The sets \eqref{sdF} and \eqref{sdC} are closed and convex, and satisfy
${\partial^F{f}(\bx)\subset\partial^C{f}(\bx)}$.
If $f$ is convex, they
reduce to the subdifferential in the sense of convex analysis: \begin{gather*}
\partial{f}(\bx):= \left\{x^\ast\in X^\ast\mid
f(x)-f(\bx)-\langle{x}^\ast,x-\bx\rangle\ge 0 \qdtx{for all} x\in X \right\}.
\end{gather*}
\if{
The relationships between the subdifferentials of $f$
at $\bar x\in\dom f$ and the corresponding normal cones to its epigraph at $(\bx, f(\bx))$ are given by
\begin{gather*}
\sd^F f(\bar x) = \left\{x^* \in X^*\mid (x^*,-1) \in N^F_{\epi f}(\bar x,f(\bar x))\right\},
\\
\partial^C{f}(\bx)= \left\{x^\ast\in X^\ast\mid
(x^*,-1)\in N_{\epi f}^C(\bx,f(\bx))\right\}.
\end{gather*}
}\fi
By convention, we set $N_{\Omega}^F(\bx)=N_{\Omega}^C(\bx):=\es$ if $\bx\notin \Omega$ and $\partial^F{f}(\bx)=\partial^C{f}(\bx):=\es$ if ${f(\bx)=+\infty}$.
It is well known that $N_{\Omega}^F(\bx)=\partial^Fi_\Omega(\bx)$ and $N_{\Omega}^C(\bx)=\partial^Ci_\Omega(\bx)$, where $i_\Omega$ is the \emph{indicator function} of $\Omega$: $i_\Omega(x)=0$ if $x\in \Omega$ and $i_\Omega(x)=+\infty$ if $x\notin \Omega$.
\sloppy

We often use the generic notations $N$ and $\sd$ for both Fr\'echet and Clarke objects, specifying wherever necessary that either $N:=N^F$ and $\sd:=\sd^F$, or $N:=N^C$ and $\sd:=\sd^C$.

The proofs of the main results in this paper rely on several
subdifferential sum rules.
Below we provide these rules for completeness; cf. \cite{Roc79,Fab89,Kru03,Mor06.1}.

\begin{lemma}[Subdifferential sum rules] \label{SR}
Suppose $X$ is a normed vector space, ${f_1,f_2:X\to\R\cup\{+\infty\}}$, and $\bx\in\dom f_1\cap\dom f_2$.
\begin{enumerate}
\item
{\bf Convex sum rule}. Suppose
$f_1$ and $f_2$ are convex and $f_1$ is continuous at a point in $\dom f_2$.
Then
$
\partial (f_1+f_2) (\bar x) = \sd f_1(\bx) +\partial f_2(\bx).
$

\item
{\bf Clarke--Rockafellar sum rule}. Suppose
$f_1$ is Lipschitz continuous and
$f_2$
is lower semicontinuous in a neighbourhood of $\bar x$.
Then
$
\partial^C(f_1+f_2)(\bar x)\subset\sd^C f_1(\bx) +\partial^Cf_2(\bx).
$

\item
{\bf Fuzzy sum rule}. Suppose $X$ is Asplund,
$f_1$ is Lipschitz continuous and
$f_2$
is lower semicontinuous in a neighbourhood of $\bar x$.
Then, for any $x^*\in\partial^F(f_1+f_2) (\bar x)$ and $\varepsilon>0$, there exist $x_1,x_2\in X$ with $\|x_i-\bar x\|<\varepsilon$, $|f_i(x_i)-f_i(\bar x)|<\varepsilon$ $(i=1,2)$, such that
$
x^*\in \partial^Ff_1(x_1) +\partial^Ff_2(x_2) + \varepsilon\B^\ast.
$
\end{enumerate}
\end{lemma}

Recall that a Banach space is \emph{Asplund} if every continuous convex function on an open convex set is Fr\'echet differentiable on a dense subset \cite{Phe93}, or equivalently, if the dual of each its separable subspace is separable.
We refer the reader to \cite{Phe93,Mor06.1} for discussions about and characterizations of Asplund spaces.
All reflexive, particularly, all finite dimensional Banach spaces are Asplund.

Another useful subdifferential
is the \emph{approximate
$G$-subdifferential} of Ioffe \cite{Iof17}.
It can replace the Clarke one in the statements of the current paper.

The following simple fact is an immediate consequence of the definitions.

\begin{lemma}\label{L2.4}
Let $\Omega_1,\Omega_2$ be subsets of a normed vector space $X$ and $\omega_i \in \Omega_i $ $(i=1,2)$.
Then
$N_{\Omega_1 \times \Omega_2}(\omega_1,\omega_2) =
N_{\Omega_1}(\omega_1)\times N_{\Omega_2}(\omega_2)$,
where in both parts of the equality $N$ stands for either the \Fr\ ($N:=N^F$) or the Clarke ($N:=N^C$) normal cone.
\end{lemma}

We are going to use a representation of the subdifferential of the function \eqref{f0} given in the next lemma; cf. \cite{KruLukTha17,CuoKru20.2}.

\begin{lemma}\label{L6}
Suppose $X$ is a normed vector space, $f:X^n\to\R_+$ is given by \eqref{f0}, $a_1,\ldots,a_{n-1},x_1,\ldots,x_n\in X$ and $\max_{1\le i\le n-1} \|x_i-a_i-x_n\|>0$.
Then
\begin{multline}\label{L6-2}
\sd f(x_1,\ldots,x_n)=\Big\{\left(x_{1}^*,\ldots, x_{n}^*\right)\in(X^*)^{n}\mid
\sum_{i=1}^{n}x_{i}^*=0,
\\
\sum_{i=1}^{n-1}\|x_{i}^*\|=1,\;
\sum_{i=1}^{n-1}\ang{x_{i}^*,x_{i}-a_i-x_n} =\max_{1\le{i}\le{n-1}} \norm{x_{i}-a_i-x_n}\Big\}.
\end{multline}
\end{lemma}

\begin{remark}\label{R5}
\begin{enumerate}
\item
It is easy to notice that in the representation \eqref{L6-2}, for any $i=1,\ldots,n-1$, either $\ang{x_{i}^*,x_{i}-a_i-x_n} =\max_{1\le{j}\le{n-1}} \norm{x_{j}-a_j-x_n}$ or ${x_{i}^*=0}$.

\item
The maximum norm on $X^{n-1}$ used in \eqref{f0} and \eqref{L6-2} is a composition of the given norm on $X$ and the maximum norm on $\R^{n-1}$.
The corresponding dual norm produces the sum of the norms in \eqref{L6-2}.
Any other finite dimensional norm can replace the maximum norm in \eqref{f0} and \eqref{L6-2} as long as the corresponding dual norm is used to replace the sum in~\eqref{L6-2}.
\end{enumerate}
\end{remark}

\section{Slope Characterizations of Non-Intersection Properties}\label{Slope}

In this and the next section, we
formulate
primal (slope) and dual (normal cone) necessary conditions for the key non-intersection properties \eqref{D1-1}, \eqref{D1-3}, \eqref{P10-1} and \eqref{P10-3}, with fixed vectors $a_i$'s.
These ubiquitous properties are present in one form or another in all parts of Definition~\ref{D8} and Proposition~\ref{P10}, as well as all known extensions of the extremality/stationarity properties.

Metric and dual necessary conditions for \eqref{D1-1} and  \eqref{P10-1} have been studied in \cite{BuiKru19}.
Here we aim at establishing
slope necessary conditions.
The nonlinearity in our model is determined by a continuous strictly increasing function $\varphi:\mathbb{R}_+ \rightarrow\mathbb{R}_+$ satisfying $\varphi(0)=0$ and $\lim_{t\to+\infty}\varphi(t)=+\infty$.
The family of all such functions is denoted by $\mathcal{C}$.
We denote by $\mathcal{C}^1$ the subfamily of functions $\varphi\in\mathcal{C}$ which are continuously differentiable on $]0,+\infty[$ with $\varphi'(t)>0$ for all $t>0$.
Obviously, if $\varphi\in\mathcal{C}$ ($\varphi\in\mathcal{C}^1$), then $\varphi\iv\in\mathcal{C}$ ($\varphi\iv\in\mathcal{C}^1$).

Along with the conventional maximum norm, we are going to consider on the product $X^n$
the following
parametric norm depending on numbers $\la>0$ and $\eta>0$:
\begin{equation}\label{pnorm}
\norm{(u_1,\ldots,u_n)}_{\la,\eta}:= \max\left\{\la\iv\norm{u_1},\ldots,\la\iv \norm{u_{n-1}},\eta\iv\norm{u_n}\right\},
\quad
u_1,\ldots,u_n\in X.
\end{equation}

To quantify non-intersection properties, we are going to use the following asymmetric distance-like quantity (\emph{nonintersect index} \cite{ZheNg11}):
\begin{gather}\label{d}
d(\Omega_1,\ldots,\Omega_{n}) :=\inf_{u_i\in\Omega_i\;(i=1,\ldots,n)}\max_{1\le{i}\le{n}-1} \norm{u_n-u_i}.
\end{gather}
If $\cap_{i=1}^n\Omega_i\ne\es$, then
$d(\Omega_1-a_1,\ldots,\Omega_{n-1}-a_{n-1},\Omega_{n}) \le\max_{1\le i\le n-1}\norm{a_i}$.

The next statement gives slope necessary conditions for the  non-intersection property \eqref{P10-1},
where the vectors $a_1,\ldots,a_{n-1}\in X$ are fixed.
The statement is a bit long as it actually combines three separate statement.
If $\cap_{i=1}^n\Omega_i\ne\es$, then condition \eqref{P10-1} implies that ${\max_{1\le i\le n-1}\norm{a_i}>0}$.
Note that condition \eqref{P10-1} is not symmetric: the role of the set $\Omega_n$ differs from that of the other sets $\Omega_1,\ldots,\Omega_{n-1}$.
This difference is exploited in the subsequent statements.

\begin{theorem}\label{T12}
Let $\Omega_1,\ldots,\Omega_n$ be closed subsets of a Banach space $X$,
$\bx\in\cap_{i=1}^n\Omega_i$, $a_i\in{X}$ $(i=1,\ldots,n-1)$, $\eps>0$ and $\varphi\in\mathcal{C}$.
Suppose that condition \eqref{P10-1} is satisfied, and
\begin{gather}\label{T12-1}
\varphi\Big(\max_{1\le i\le n-1}\norm{a_i}\Big) <\varphi(d(\Omega_1-a_1,\ldots,\Omega_{n-1}-a_{n-1}, \Omega_n))+\eps,
\end{gather}
where $d(\cdot)$ is given by \eqref{d}.
Then, for any $\la>0$ and $\eta>0$,
there exist points
${\omega_i\in\Omega_i\cap{B}_\la(\bx)}$ $(i=1,\ldots,n-1)$ and $\omega_n\in\Omega_n\cap \B_\eta(\bx)$ such that
\begin{gather}\label{T12-2}
\hspace{-2mm} \sup_{\substack{u_i\in\Omega_i\;(i=1,\ldots,n)\\ (u_1,\ldots,u_n)\ne(\omega_1,\ldots,\omega_n)}} \hspace{-1mm} \frac{\varphi\Big(\max\limits_{1\le{i}\le{n}-1} \norm{\omega_n+a_i-\omega_i}\Big) -\varphi\Big(\max\limits_{1\le{i}\le{n}-1} \norm{u_n+a_i-u_i}\Big)} {\norm{(u_1-\omega_1,\ldots,u_n-\omega_n)}_{\la,\eta}} <\eps,
\\\label{T12-3}
0<\max_{1\le{i}\le{n}-1}\|\omega_n+a_i-\omega_i\|\le \max_{1\le i\le n-1}\norm{a_i}.
\end{gather}
As a consequence,
\begin{gather}\label{T12-4}
\hspace{-1mm} \limsup_{\substack{\Omega_i\ni u_i\to\omega_i\; (i=1,\ldots,n)\\ (u_1,\ldots,u_n)\ne(\omega_1,\ldots,\omega_n)}} \hspace{-1mm} \frac{\varphi\Big(\max\limits_{1\le{i}\le{n}-1} \norm{\omega_n+a_i-\omega_i}\Big) -\varphi\Big(\max\limits_{1\le{i}\le{n}-1} \norm{u_n+a_i-u_i}\Big)} {\norm{(u_1-\omega_1,\ldots,u_n-\omega_n)}_{\la,\eta}} <\eps.
\end{gather}
Moreover, if $\varphi$ is differentiable at $\max_{1\le{i}\le{n}-1} \norm{\omega_n+a_i-\omega_i}$, then
\begin{multline}\label{T12-5}
\varphi'\Big(\max\limits_{1\le{i}\le{n}-1} \norm{\omega_n+a_i-\omega_i}\Big)
\\
\times\limsup_{\substack{\Omega_i\ni u_i\to\omega_i\; (i=1,\ldots,n)\\ (u_1,\ldots,u_n)\ne(\omega_1,\ldots,\omega_n)}} \frac{\max\limits_{1\le{i}\le{n}-1} \norm{\omega_n+a_i-\omega_i} -\max\limits_{1\le{i}\le{n}-1} \norm{u_n+a_i-u_i}} {\norm{(u_1-\omega_1,\ldots,u_n-\omega_n)}_{\la,\eta}} <\eps,
\end{multline}
with the convention $0\cdot(+\infty)=0$.
\end{theorem}

The short proof below encapsulates the traditional arguments used in the primal space part of the proof of all extremality/stationarity statements for collections of sets.

\begin{proof}
Let $\la>0$, $\eta>0$, and a number $\eps'$ satisfy
\begin{gather}\label{T12P1}
\varphi\Big(\max_{1\le i\le n-1}\norm{a_i}\Big) -\varphi(d(\Omega_1-a_1,\ldots,\Omega_{n-1}-a_{n-1},\Omega_n))<\eps'<\eps.
\end{gather}
Consider a continuous function $f:X^n\to\R_+$:
\begin{gather}\label{f}
f(u_1,\ldots,u_{n}):=\varphi\Big(\max\limits_{1\le{i}\le{n}-1} \norm{u_n+a_i-u_i}\Big),
\quad
u_1,\ldots,u_n\in X.
\end{gather}
It follows from \eqref{P10-1} and \eqref{T12P1} that
\begin{gather*}
\label{C4.P3P4}
f(\bx,\ldots,\bx)=\varphi\Big(\max_{1\le i\le n-1}\norm{a_i}\Big) <\varphi(d(\Omega_1-a_1,\ldots,\Omega_{n-1}-a_{n-1},\Omega_n))+\eps',
\end{gather*}
and condition \eqref{12} is satisfied.
Applying
Lemma~\ref{evp}
to the restriction of the function $f$ to the complete metric space $\Omega_1\times\ldots\times \Omega_n$ with the metric induced by the norm \eqref{pnorm},
we find points $\omega_i\in \Omega_i\cap \B_\la(\bx)$ ($i=1,\ldots,n-1$) and $\omega_n\in\Omega_n\cap \B_\eta(\bx)$ such that
\begin{gather*}
0<f(\omega_{1},\ldots,\omega_n)\le f(\bx,\ldots,\bx),\\
f(u_1,\ldots,u_n)+\eps' \norm{(u_1-\omega_1,\ldots,u_n-\omega_n)}_{\la,\eta}\ge f(\omega_1,\ldots,\omega_n)
\end{gather*}
for all $u_i\in \Omega_i$ ($i=1,\ldots,n$).
In view of the monotonicity of $\varphi$, the last two inequalities imply conditions \eqref{T12-2} and \eqref{T12-3}.
Condition \eqref{T12-2} obviously yields \eqref{T12-4}.
If $\varphi$ is differentiable at $\max_{1\le{i}\le{n}-1} \norm{\omega_n+a_i-\omega_i}$, then condition \eqref{T12-4} implies \eqref{T12-5} thanks to Lemma~\ref{L2}.
\end{proof}

\begin{remark}\label{R13}
\begin{enumerate}
\item
The expressions in the \LHS s of the inequalities \eqref{T12-2} and \eqref{T12-4} are the main ingredients of the, respectively, global and local slopes of the restriction of the function $f$ given by \eqref{f} to the complete metric space ${\Omega_1\times\ldots\times \Omega_n}$ with the metric induced by the norm \eqref{pnorm} (cf. \cite{Kru15}). 
This observation justifies the name `slope conditions' adopted in the current paper for this type of estimates as well as the reference to Lemma~\ref{L2} in the proof of Theorem~\ref{T12}.
\item
The functions in the \LHS s of the inequalities \eqref{T12-2}, \eqref{T12-4} and \eqref{T12-5} are computed at some points
$\omega_i\in\Omega_i$ $(i=1,\ldots,n)$ in a \nbh\ of the reference point $\bx$.
The size of the \nbh s is controlled by the parameters $\la$ and $\eta$, which can be chosen arbitrarily small.
However, decreasing these parameters weakens  conditions \eqref{T12-2}, \eqref{T12-4} and \eqref{T12-5}.
Note also that the \nbh s for $\omega_i$ $(i=1,\ldots,n-1)$ on one hand, and $\omega_n$ on the other hand are controlled by different parameters.
This reflects the fact that condition \eqref{P10-1} is not symmetric.
\item
It is easy to see that the conditions under $\sup$ in \eqref{T12-2} and under $\limsup$ in \eqref{T12-4} and \eqref{T12-5} can be complemented by the inequality
$\max_{1\le{i}\le{n}-1} \norm{u_n+a_i-u_i}<\max_{1\le{i}\le{n}-1} \norm{\omega_n+a_i-\omega_i}$
with the convention that the supremum over the empty set equals $0$.
\item
Condition \eqref{T12-3} relates points $\omega_{1},\ldots,\omega_{n}$ to the given vectors $a_1,\ldots,a_n$ and complements the other conditions in Theorem~\ref{T12} on the choice of these points.
The smaller
these vectors are, the more binding condition \eqref{T12-3} is; for instance, it follows from conditions $\omega_i\in{B}_\la(\bx)$ $(i=1,\ldots,n-1)$ and $\omega_n\in \B_\eta(\bx)$ that, for each $i=1,\ldots,n-1$, $\|\omega_n-\omega_i\|<\la+\eta$, while condition \eqref{T12-3} gives an alternative estimate: $\|\omega_n-\omega_i\|<2\max_{1\le i \le n-1}\norm{a_i}$, which obviously `outperforms' the first one when $\max_{1\le i \le n-1}\norm{a_i}<(\la+\eta)/2$.
\sloppy
\item
Condition \eqref{T12-1} in Theorem~\ref{T12} can obviously be replaced by a simpler (though stronger!) condition $\varphi\left(\max_{1\le i\le n-1}\norm{a_i}\right)<\eps$.
This weakened version of Theorem~\ref{T12} is sufficient for characterizing the conventional extremality/stationarity properties in Definition~\ref{D8}.
One can go even further and require simply that $\max_{1\le i\le n-1}\norm{a_i}<\eps$.
Of course, in this case $\eps$ in the inequalities \eqref{T12-2}, \eqref{T12-4} and \eqref{T12-5} must be replaced by $\varphi(\eps)$.
This creates an interesting phenomenon: $\varphi$ disappears completely from the assumptions of Theorem~\ref{T12} and remains only in its conclusions (which must hold true for any $\varphi\in\mathcal{C}$!)
The importance of the full version of a condition of the type \eqref{T12-1} for some applications was demonstrated in \cite{ZheNg06,ZheNg11}.
\end{enumerate}
\end{remark}

Theorem~\ref{T12} gives
primal (slope) necessary conditions for the asymmetric non-intersection property \eqref{P10-1}, which is a special case of the
key
property \eqref{D1-1} in the definition of extremality.
Now observe that necessary conditions for even more general than \eqref{D1-1} symmetric `local' non-intersection property \eqref{D1-3} can be straightforwardly deduced from Theorem~\ref{T12}.
It is sufficient to add to the given collection of $n$ sets a closed ball $\overline{\B}_\eta(\bx)\subset \B_\rho(\bx)$ and apply Theorem~\ref{T12}.

\begin{theorem}\label{T14}
Let $\Omega_1,\ldots,\Omega_n$ be closed subsets of a Banach space $X$,
$\bx\in\cap_{i=1}^n\Omega_i$, $a_i\in{X}$ $(i=1,\ldots,n)$, $\eps>0$, $\rho\in]0,+\infty]$ and $\varphi\in\mathcal{C}$.
Suppose that condition \eqref{D1-3} is satisfied and
\begin{gather}\label{T14-1}
\varphi\Big(\max_{1\le i\le n}\norm{a_i}\Big) <\varphi(d(\Omega_1-a_1,\ldots,\Omega_{n}-a_{n}, {B}_\eta(\bx)))+\eps,
\end{gather}
where $d(\cdot)$ is given by \eqref{d}.
Then, for any $\la>0$ and $\eta\in]0,\rho[$,
there exist points
$\omega_i\in\Omega_i\cap{B}_\la(\bx)$ $(i=1,\ldots,n)$ and $x\in \B_\eta(\bx)$ such that
\begin{gather}\label{T14-2}
\sup_{\substack{u_i\in\Omega_i\;(i=1,\ldots,n),\,u\in \B_\eta(\bx)\\ (u_1,\ldots,u_n,u)\ne(\omega_1,\ldots,\omega_n,x)}} \frac{\varphi\Big(\max\limits_{1\le{i}\le{n}} \norm{x+a_i-\omega_i}\Big) -\varphi\Big(\max\limits_{1\le{i}\le{n}} \norm{u+a_i-u_i}\Big)} {\norm{(u_1-\omega_1,\ldots,u_n-\omega_n,u-x)}_{\la,\eta}} <\eps,
\\\label{T14-3}
0<\max_{1\le{i}\le{n}}\|x+a_i-\omega_i\|\le \max_{1\le i\le n}\norm{a_i}.
\end{gather}
As a consequence,
\begin{gather}\label{T14-4}
\limsup_{\substack{\Omega_i\ni u_i\to\omega_i\; (i=1,\ldots,n),\,u\to x\\ (u_1,\ldots,u_n,u)\ne(\omega_1,\ldots,\omega_n,x)}} \frac{\varphi\Big(\max\limits_{1\le{i}\le{n}} \norm{x+a_i-\omega_i}\Big) -\varphi\Big(\max\limits_{1\le{i}\le{n}} \norm{u+a_i-u_i}\Big)} {\norm{(u_1-\omega_1,\ldots,u_n-\omega_n,u-x)}_{\la,\eta}} <\eps.
\end{gather}
Moreover, if $\varphi$ is differentiable at $\max_{1\le{i}\le{n}} \norm{x+a_i-\omega_i}$, then
\begin{multline}\label{T14-5}
\varphi'\Big(\max\limits_{1\le{i}\le{n}} \norm{x+a_i-\omega_i}\Big)
\\
\times\limsup_{\substack{\Omega_i\ni u_i\to\omega_i\; (i=1,\ldots,n),\,u\to x\\ (u_1,\ldots,u_n,u)\ne(\omega_1,\ldots,\omega_n,x)}} \frac{\max\limits_{1\le{i}\le{n}} \norm{x+a_i-\omega_i} -\max\limits_{1\le{i}\le{n}} \norm{u+a_i-u_i}} {\norm{(u_1-\omega_1,\ldots,u_n-\omega_n,u-x)}_{\la,\eta}} <\eps,
\end{multline}
with the convention $0\cdot(+\infty)=0$.
\end{theorem}

\begin{remark}
The comments concerning Theorem~\ref{T12} made in Remark~\ref{R13} are, with obvious modifications, applicable to Theorem~\ref{T14} and the subsequent statements in this paper.
\end{remark}

All the properties in Definition~\ref{D8} as well as the
primal necessary conditions for the non-intersection properties in Theorems~\ref{T12} and \ref{T14} presume that the sets have a common point.
Fortunately Theorem~\ref{T12} is rich enough to characterize a non-intersection property without this assumption.
Indeed, if
$\cap_{i=1}^n\Omega_i=\emptyset$,
then, for any points
$\omega_i\in\Omega_i$ ${(i=1,\ldots,n)}$, one can consider the sets $\Omega_i':=\Omega_i-\omega_i$ ${(i=1,\ldots,n)}$, which obviously satisfy $0\in\cap_{i=1}^n \Omega_i'$.
Moreover, after setting $a_i:=\omega_{n}-\omega_{i}$ (${i=1,\ldots,n-1}$), one has
\sloppy
\begin{gather}\label{23}
\bigcap_{i=1}^{n-1}(\Omega_i'-a_i)\cap\Omega_{n}'= \bigcap_{i=1}^{n-1}(\Omega_i-\omega_{n})\cap (\Omega_n-\omega_n)=\bigcap_{i=1}^{n}\Omega_i-\omega_{n}
=\emptyset.
\end{gather}
Thus, Theorem~\ref{T12} is applicable and we immediately arrive at the next statement.
Note that
$d(\Omega'_1-a_1,\ldots,\Omega'_{n-1}-a_{n-1},\Omega'_{n}) =d(\Omega_1,\ldots,\Omega_{n})$.

\begin{proposition}
Let $\Omega_1,\ldots,\Omega_n$ be closed subsets of a Banach space $X$, $\omega_i\in\Omega_i$ ${(i=1,\ldots,n)}$, $\eps>0$ and $\varphi\in\mathcal{C}$.
Suppose that $\cap_{i=1}^n\Omega_i=\emptyset$ and
\begin{gather}\notag
\varphi\Big(\max_{1\le i\le n-1} \norm{\omega_n-\omega_i}\Big) <\varphi(d(\Omega_1,\ldots,\Omega_n))+\eps
\end{gather}
where $d(\cdot)$ is given by \eqref{d}.
Then, for any $\la>0$ and $\eta>0$,
there exist points
$\omega_i'\in\Omega_i\cap{B}_\la(\omega_i)$ ${(i=1,\ldots,n-1)}$ and $\omega_n'\in\Omega_n\cap \B_\rho(\omega_n)$ such that
\begin{gather}\notag
\sup_{\substack{u_i\in\Omega_i\;(i=1,\ldots,n)\\ (u_1,\ldots,u_n)\ne(\omega'_1,\ldots,\omega'_n)}} \frac{\varphi\Big(\max\limits_{1\le{i}\le{n}-1} \norm{\omega'_n-\omega'_i}\Big) -\varphi\Big(\max\limits_{1\le{i}\le{n}-1} \norm{u_n-u_i}\Big)} {\norm{(u_1-\omega'_1,\ldots,u_n-\omega'_n)}_{\la,\eta}} <\eps,
\\\label{P16-3}
0<\max_{1\le{i}\le{n}-1}\|\omega'_n-\omega'_i\|\le \max_{1\le i \le n-1}\norm{\omega_n-\omega_i}.
\end{gather}
As a consequence,
\begin{gather*}
\limsup_{\substack{\Omega_i\ni u_i\to\omega'_i\; (i=1,\ldots,n)\\ (u_1,\ldots,u_n)\ne(\omega'_1,\ldots,\omega'_n)}} \frac{\varphi\Big(\max\limits_{1\le{i}\le{n}-1} \norm{\omega'_n-\omega'_i}\Big) -\varphi\Big(\max\limits_{1\le{i}\le{n}-1} \norm{u_n-u_i}\Big)} {\norm{(u_1-\omega'_1,\ldots,u_n-\omega'_n)}_{\la,\eta}} <\eps.
\end{gather*}
Moreover, if $\varphi$ is differentiable at $\max_{1\le{i}\le{n}-1} \norm{\omega'_n-\omega'_i}$, then
\begin{gather*}
\varphi'\Big(\max\limits_{1\le{i}\le{n}-1} \norm{\omega'_n-\omega'_i}\Big) \limsup_{\substack{\Omega_i\ni u_i\to\omega'_i\; (i=1,\ldots,n)\\ (u_1,\ldots,u_n)\ne(\omega'_1,\ldots,\omega'_n)}} \frac{\max\limits_{1\le{i}\le{n}-1} \norm{\omega'_n-\omega'_i} -\max\limits_{1\le{i}\le{n}-1} \norm{u_n-u_i}} {\norm{(u_1-\omega'_1,\ldots,u_n-\omega'_n)}_{\la,\eta}} <\eps,
\end{gather*}
with the convention $0\cdot(+\infty)=0$.
\end{proposition}

\section{Dual Characterization of Non-Intersection Properties} \label{dual}

In this section, we require the function $\varphi$ to be continuously differentiable.

The dual norm on $(X^*)^{n}$ corresponding to \eqref{pnorm} has the following form:
\begin{gather}\notag
\|(x_1^*,\ldots,x_n^*)\|_{\la,\eta} =\la\sum_{i=1}^{n-1}\|x_i^*\|+\eta\|x_n^*\|,\quad
x^*_1,\ldots,x^*_n\in X^*.
\end{gather}

The next theorem
gives
dual necessary conditions for the non-intersection property \eqref{P10-1}.
It combines two statements: in general Banach spaces in terms of Clarke normals and in Asplund spaces in terms of \Fr\ normals, and
generalizes and improves \cite[Theorem~6.1]{BuiKru19}.

\begin{theorem}\label{T17}
Let $\Omega_1,\ldots,\Omega_n$ be closed subsets of a Banach space $X$,
$\bx\in\cap_{i=1}^n\Omega_i$, ${a_i\in{X}}$ $(i=1,\ldots,n-1)$, $\eps>0$ and $\varphi\in\mathcal{C}^1$.
Suppose that conditions \eqref{P10-1} and \eqref{T12-1}
are satisfied.
Then, for any $\la>0$ and $\eta>0$,
\begin{enumerate}
\item
there exist points
$\omega_i\in\Omega_i\cap{B}_\la(\bx)$ $(i=1,\ldots,n-1)$ and $\omega_n\in\Omega_n\cap \B_\eta(\bx)$ satisfying condition \eqref{T12-3}, and vectors
$x_i^*\in X^*$ $(i=1,\ldots,n)$ such that
\begin{gather}
\label{T17-1}
\sum_{i=1}^{n}x_i^*=0,\quad \sum_{i=1}^{n-1}\norm{x_i^*}=1,
\\
\label{T17-2}
\hspace{-4mm} \varphi'\Big(\max_{1\le{i}\le{n}-1} \|\omega_n+a_i-\omega_{i}\|\Big) \Big({\la}\sum_{i=1}^{n-1} d\left(x_i^*, {N}_{\Omega_i}(\omega_i)\right) + {\eta}d\left(x_n^*,{N}_{\Omega_n}(\omega_n)\right)\Big) <\eps,
\\
\label{T17-3}
\sum_{i=1}^{n-1}\ang{x_{i}^*,\omega_n+a_i-\omega_{i}} =\max_{1\le{i}\le{n}-1} \|\omega_n+a_i-\omega_{i}\|,
\end{gather}
where $N$ stands for the Clarke normal cone ${(N=N^C)}$;

\item
if $X$ is Asplund,
then, for any $\tau\in]0,1[$,
there exist points
$\omega_i\in\Omega_i\cap{B}_\la(\bx)$ ${(i=1,\ldots,n-1)}$ and $\omega_n\in\Omega_n\cap \B_\eta(\bx)$ satisfying
\begin{gather}
\label{T17-4}
0<\max_{1\le{i}\le{n}-1} \|\omega_n+a_i-\omega_{i}\| <\varphi\iv(\varphi(d(\Omega_1-a_1,\ldots,\Omega_{n-1}-a_{n-1}, \Omega_n))+\eps),
\end{gather}
and vectors
$x_i^*\in X^*$ $(i=1,\ldots,n)$ satisfying conditions \eqref{T17-1}, \eqref{T17-2}, with $N$ standing for the \Fr\ normal cone ${(N=N^F)}$, and
\begin{gather}
\label{T17-5}
\sum_{i=1}^{n-1}\ang{x_{i}^*,\omega_n+a_i-\omega_{i}} >\tau\max_{1\le{i}\le{n}-1} \|\omega_n+a_i-\omega_{i}\|.
\end{gather}
\end{enumerate}
\end{theorem}

The proof below
employs the primal necessary conditions in Theorem~\ref{T12}, and
encapsulates the traditional arguments involving appropriate subdifferential sum rules, used in the dual space part of the proof of all extremality/stationarity statements for collections of sets.

\begin{proof}
Choose a number $\eps'$ satisfying condition \eqref{T12P1}.
By Theorem~\ref{T12},
there exist points
$\omega_i\in\Omega_i\cap{B}_\la(\bx)$ $(i=1,\ldots,n-1)$ and $\omega_n\in\Omega_n\cap \B_\eta(\bx)$ satisfying condition \eqref{T12-3} such that condition \eqref{T12-5} holds with $\eps'$ in place of $\eps$.
The last condition yields
\begin{equation}
\label{T17P1}
0\in\partial(f+f_1+f_2)(\omega_1,\ldots,\omega_n),
\end{equation}
where ${\sd:=\sd^F}$, $f$ is given by \eqref{f0} and, for all $u_1,\ldots,u_n\in X$,
\begin{align*}
f_1(u_1,\ldots,u_{n}):=
&\frac{\eps'} {\varphi'(f(\omega_1,\ldots,\omega_{n}))} \norm{(u_1- \omega_1,\ldots,u_n- \omega_n)}_{\la,\eta},
\\\notag
f_2(u_1,\ldots,u_{n}):=&
\begin{cases}
0&\mbox{if } u_i\in\Omega_i\; (i=1,\ldots,n),
\\
+\infty&\mbox{otherwise}.
\end{cases}
\end{align*}
Functions $f$ and $f_1$ are convex and Lipschitz continuous, and $f_2$ is lower semicontinuous.
From this point we split the proof into two cases.

(i)
$X$ is a general Banach space.
Condition \eqref{T17P1} obviously holds with ${\sd:=\sd^C}$.
By
Lemma~\ref{SR}(ii), there exist three subgradients:  $(v_{1}^*,\ldots,v_{n}^*)\in \sd{f}(\omega_{1},\ldots,\omega_{n})$, $(v_{11}^*,\ldots,v_{1n}^*)\in \sd{f}_1(\omega_{1},\ldots,\omega_{n})$ and ${(v_{21}^*,\ldots,v_{2n}^*)\in \sd^C{f}_2(\omega_{1},\ldots,\omega_{n})}$ such that
$v_{i}^*+v_{1i}^*+v_{2i}^*=0$ ${(i=1,\ldots,n)}$.
Observe that $f(\omega_{1},\ldots,\omega_{n})>0$
(thanks to \eqref{T12-3}), and
$f_2$ is the indicator function of the set $\Omega_1\times\ldots\times\Omega_n$. Hence, by Lemmas~\ref{L2.4} and \ref{L6},
\begin{gather}
\label{T17P3}
\sum_{i=1}^{n}v_{i}^*=0,
\quad
\sum_{i=1}^{n-1}\|v_{i}^*\|=1,
\\\label{T17P4}
\sum_{i=1}^{n-1}\ang{v_{i}^*,\omega_{i}-a_i-\omega_n} =\max_{1\le{i}\le{n-1}} \norm{\omega_{i}-a_i-\omega_n},
\\\label{T17P5}
\varphi'(f(\omega_1,\ldots,\omega_{n})) \Big(\la\sum_{i=1}^{n-1}\norm{v_{1i}^*} +{\eta}\norm{v_{1n}^*}\Big) \le\eps',
\end{gather}
and
$v_{2i}^*\in{N}^C_{\Omega_i}(\omega_i)$ $(i=1,\ldots,n)$.
Setting
$x_i^*:=-v_{i}^*$ $(i=1,\ldots,{n})$, we immediately obtain conditions \eqref{T17-1} and \eqref{T17-3}.
Moreover,
\begin{align*}
{\la}\sum_{i=1}^{n-1} d\left(x_i^*, {N}^C_{\Omega_i}(\omega_i)\right)+ {\eta}d\left(x_n^*,{N}^C_{\Omega_n}(\omega_n)\right)
&\le{\la}\sum_{i=1}^{n-1}\|v_{i}^*+v^*_{2i}\|+ {\eta}\|v_{n}^*+v^*_{2n}\|
\\
&={\la}\sum_{i=1}^{n-1}\|v_{1i}^*\|+ {\eta}\norm{v^*_{1n}},
\end{align*}
and condition \eqref{T17-2} is a consequence of \eqref{T17P5}.

(ii)
Let $X$ be Asplund, and
$\tau\in]0,1[$.
In view of \eqref{T17P1}, we can apply
Lemma~\ref{SR}(iii) to the sum of $f+f_1$ and $f_2$ followed by
Lemma~\ref{SR}(i) applied to the sum of $f$ and $f_1$:
for any ${\xi>0}$, there are points $x_{i}\in X$ and $\omega_i'\in\Omega_i$ $(i=1,\ldots,n)$
and subgradients $(v_{1}^*,\ldots,v_{n}^*)\in \sd{f}(x_{1},\ldots,x_{n})$, $(v_{11}^*,\ldots,v_{1n}^*)\in \sd{f}_1(x_{1},\ldots,x_{n})$ and $(v_{21}^*,\ldots,v_{2n}^*)\in \sd^F{f}_2(\omega_{1}',\ldots,\omega_{n}')$ such that
\begin{gather*}
\max_{1\le i\le n}\|x_i-\omega_i\|<\xi,\quad \max_{1\le i\le n}\|\omega_i'-\omega_i\|<\xi,
\quad
\sum_{i=1}^{n} \norm{v_{i}^*+v_{1i}^*+v_{2i}^*}<\xi.
\end{gather*}
The number $\xi$ can be chosen small enough so that $\omega_i'\in \B_{\la}(\bar x)$ $(i=1,\ldots,n-1)$, ${\omega_n'\in \B_{\eta}(\bar x)}$, ${f}(x_{1},\ldots,x_{n})>0$, ${f}(\omega'_{1},\ldots,\omega'_{n})>0$, conditions \eqref{T17-4} and \eqref{T17-5} are satisfied with $\omega_i'$ in place of $\omega_i$ $(i=1,\ldots,n)$, and, taking into account the continuity of $f$ and~$\varphi'$,
\begin{equation*}
\frac{\eps'}{\varphi'\left(f(x_1,\ldots,x_{n})\right)} +\max\{\la,\eta\}\xi <\frac{\eps} {\varphi'\left(f(\omega'_1,\ldots,\omega'_{n})\right)}.
\end{equation*}
By Lemmas~\ref{L2.4} and \ref{L6}, vectors $(v_{1}^*,\ldots,v_{n}^*)$ and $(v_{11}^*,\ldots,v_{1n}^*)$ satisfy conditions \eqref{T17P3}, \eqref{T17P4} and \eqref{T17P5} with $x_i$ in place of $\omega_i$ $(i=1,\ldots,n)$, and
$v_{2i}^*\in{N}^F_{\Omega_i}(\omega_i')$ $(i=1,\ldots,n)$.
Set $x_i^*:=-v_{i}^*$ $(i=1,\ldots,{n})$.
Conditions \eqref{T17-1} follow immediately.
Moreover,
\begin{align*}
{\la}\sum_{i=1}^{n-1} d\left(x_i^*, {N}^F_{\Omega_i}(\omega_i')\right) &+ {\eta}d\left(x_n^*,{N}^F_{\Omega_n}(\omega_n')\right)
\le{\la}\sum_{i=1}^{n-1}\|v_{i}^*+v^*_{2i}\|+ {\eta}\|v_{n}^*+v^*_{2n}\|
\\
&\le{\la}\sum_{i=1}^{n-1}\|v_{1i}^*\|+ {\eta}\norm{v^*_{1n}}+\max\{\la,\eta\}\sum_{i=1}^{n} \norm{v_{i}^*+v_{1i}^*+v_{2i}^*}
\\
&\le\frac{\eps'} {\varphi'\left(f(x_1,\ldots,x_{n})\right)}+\max\{\la,\eta\}\xi <\frac{\eps} {\varphi'\left(f(\omega'_1,\ldots,\omega'_{n})\right)},
\end{align*}
i.e. condition \eqref{T17-2} is satisfied with $\omega_i'$ in place of $\omega_i$ $(i=1,\ldots,n)$.
\end{proof}

\begin{remark}\label{R18}
\begin{enumerate}
\item
Inequality \eqref{T17-2} together with the first equality in \eqref{T17-1} play the key role in asymmetric dual necessary conditions for extremality/stationarity properties.
The inequality ensures that the dual vectors $x_1^*,\ldots,x_n^*$, whose sum is zero, are close to the corresponding normal cones.
The second equality in \eqref{T17-1} is the normalization condition for the collection of dual vectors; it ensures that the conditions are nontrivial.

\item
When $\varphi$ is linear, the \LHS\ of \eqref{T17-2} is independent of the vectors $a_1,\ldots,a_{n-1}$.

\item
Conditions \eqref{T17-3} and \eqref{T17-5} first appeared explicitly in \cite{ZheNg06} and were explored further in \cite{ZheNg11,BuiKru19}.
Conditions of this type relate dual vectors $x_i^*$ and primal space vectors $\omega_{n}+a_i-\omega_{i}$ ($i=1,\ldots,n-1$), and allow one to reduce the number of dual vectors involved in checking dual necessary conditions for extremality/stationarity properties.
Such conditions also play an important role in characterizations of intrinsic transversality \cite{ThaBuiCuoVer20}.

\item
Primal
conditions \eqref{T12-3} and \eqref{T17-4} provide additional necessary conditions for the non-intersection properties (cf. Remark~\ref{R13}(iv)).
They have not been used in this context before.

\item
Condition \eqref{T17-2} with \Fr\ normal cones is obviously stronger than its version with Clarke normal cones.
On the other hand, conditions \eqref{T17-4} and \eqref{T17-5} in the second part of Theorem~\ref{T17} are weaker than the corresponding conditions \eqref{T12-3} and \eqref{T17-3} in its first part.
This is because of the fuzzy sum rule used in its proof.

\item
Clarke normal cones in part (i) of Theorem~\ref{T17} and the other dual space characterizations in this paper can be replaced by Ioffe's \emph{$G$-normal cones} \cite{Iof17}. 
\end{enumerate}
\end{remark}

As in Section~\ref{Slope}, necessary conditions for the more general than \eqref{P10-1} symmetric local non-inter\-section property \eqref{D1-3} can be straightforwardly deduced from Theorem~\ref{T17} by using the same simple trick:
adding to the given collection of $n$ sets a closed ball $\overline{\B}_\eta(\bx)\subset \B_\rho(\bx)$.
The next statement generalizes and improves \cite[Theorem~6.3]{BuiKru19}.

\begin{theorem}\label{T19}
Let $\Omega_1,\ldots,\Omega_n$ be closed subsets of a Banach space $X$,
$\bx\in\cap_{i=1}^n\Omega_i$, ${a_i\in{X}}$ ${(i=1,\ldots,n)}$, $\rho\in]0,+\infty]$, $\eps>0$ and $\varphi\in\mathcal{C}^1$.
Suppose that conditions \eqref{D1-3} and \eqref{T14-1}
are satisfied.
Then, for any $\la>0$ and $\eta\in]0,\rho[$,
\begin{enumerate}
\item
there exist points
$\omega_i\in\Omega_i\cap{B}_\la(\bx)$ $(i=1,\ldots,n)$ and $x\in \B_\eta(\bx)$ satisfying condition \eqref{T14-3}, and vectors
$x_i^*\in X^*$ $(i=1,\ldots,n)$ satisfying $\sum_{i=1}^{n}\norm{x_i^*}=1$ and
\begin{gather}
\label{T19-2}
\varphi'\left(\max_{1\le{i}\le{n}} \|x+a_i-\omega_{i}\|\right) \left({\la}\sum_{i=1}^{n} d\left(x_i^*, {N}_{\Omega_i}(\omega_i)\right) + {\eta}\norm{\sum_{i=1}^n x_i^*}\right)<\eps,
\\
\label{T19-3}
\sum_{i=1}^{n}\ang{x_{i}^*,x+a_i-\omega_{i}} =\max_{1\le{i}\le{n}} \|x+a_i-\omega_{i}\|,
\end{gather}
where $N$ stands for the Clarke normal cone ${(N=N^C)}$;
\item
if $X$ is Asplund,
then, for any $\tau\in]0,1[$,
there exist points
$\omega_i\in\Omega_i\cap{B}_\la(\bx)$ $(i=1,\ldots,n)$ and $x\in \B_\eta(\bx)$ satisfying
\begin{gather}
\label{T19-4}
0<\max_{1\le{i}\le{n}} \|x+a_i-\omega_{i}\| <\varphi\iv(\varphi(d(\Omega_1-a_1,\ldots,\Omega_{n}-a_{n}, {B}_\eta(\bx)))+\eps),
\end{gather}
and vectors
$x_i^*\in X^*$ $(i=1,\ldots,n)$ satisfying $\sum_{i=1}^{n}\norm{x_i^*}=1$, condition \eqref{T19-2} with $N$ standing for the \Fr\ normal cone ${(N=N^F)}$, and
\begin{gather}
\label{T19-5}
\sum_{i=1}^{n}\ang{x_{i}^*,x+a_i-\omega_{i}} >\tau\max_{1\le{i}\le{n}}
\|x+a_i-\omega_{i}\|.
\end{gather}
\end{enumerate}
\end{theorem}

\begin{proof}
The statement is a direct consequence of Theorem~\ref{T17} applied to the collection of $n+1$ closed sets $\Omega_1,\ldots,\Omega_n,\overline{\B}_\eta(\bx)$.
It is sufficient to notice that, once $x\in \B_\eta(\bx)$, we have $N_{\B_\eta(\bx)}(x)=\{0\}$, and consequently, $d\big(-\sum_{i=1}^n x_i^*,N_{\B_\eta(\bx)}(x)\big)=\norm{\sum_{i=1}^n x_i^*}$.
\end{proof}

\begin{remark}
The comments concerning Theorem~\ref{T17} made in Remark~\ref{R18} are with obvious modifications applicable to Theorem~\ref{T19} and the subsequent statements in this paper.
\end{remark}

The single common point $\bx\in\cap_{i=1}^n\Omega_i$ in Theorems~\ref{T17} and \ref{T19} can be replaced by a collection of individual points $\omega_i\in\Omega_i$ $(i=1,\ldots,n)$ (which always exist as long as the sets are nonempty).
The next statement is a consequence of Theorem~\ref{T17} applied to the collection of sets $\Omega_i':=\Omega_i-\omega_i$ ${(i=1,\ldots,n)}$, which obviously have a common point $0\in\cap_{i=1}^n \Omega_i'$.
It generalizes and improves \cite[Corollary 6.1]{BuiKru19}.

\begin{proposition}\label{P21}
Let $\Omega_1,\ldots,\Omega_n$ be closed subsets of a Banach space $X$,
$\omega_i\in\Omega_i$ ${(i=1,\ldots,n)}$, $a_i\in{X}$ ${(i=1,\ldots,n-1)}$, $\eps>0$ and $\varphi\in\mathcal{C}^1$.
Suppose that
\begin{gather}\notag
\bigcap\limits_{i=1}^{n-1}(\Omega_i-\omega_i-a_i) \cap(\Omega_{n}-\omega_n)
=\emptyset,
\\\notag
\varphi\Big(\max_{1\le i\le n-1}\norm{a_i}\Big) <\varphi(d(\Omega_1-\omega_i-a_1,\ldots, \Omega_{n-1}-\omega_{n-1}-a_{n-1},\Omega_n-\omega_n))+\eps.
\end{gather}
Set $a'_i:=a_i+\omega_i-\omega_{n}$ $(i=1,\ldots,n-1)$.
Then, for any $\la>0$, $\eta>0$,
\begin{enumerate}
\item
there exist points
$\omega'_i\in\Omega_i\cap{B}_\la(\omega_i)$ $(i=1,\ldots,n-1)$ and $\omega'_n\in\Omega_n\cap \B_\eta(\omega_n)$ satisfying
\begin{gather}
\label{P21-3}
0<\max_{1\le i \le n-1} \norm{\omega_{n}'+a'_i-\omega_i'}\le \max_{1\le i \le n-1}\norm{a_i},
\end{gather}
and vectors
$x_i^*\in X^*$ $(i=1,\ldots,n)$ satisfying conditions \eqref{T17-1} and
\begin{multline}\label{P21-4}
\varphi'\left(\max_{1\le i \le n-1} \norm{\omega_{n}'+a'_i-\omega_i'}\right)
\\
\times\left({\la}\sum_{i=1}^{n-1} d\left(x_i^*, {N}_{\Omega_i}(\omega'_i)\right) + {\eta}d\left(x_n^*,{N}_{\Omega_n}(\omega'_n)\right)\right) <\eps,
\end{multline}
\vspace{-6mm}
\begin{gather*}
\sum_{i=1}^{n-1} \ang{x_i^*,\omega_{n}'+a'_i-\omega_i'} =\max_{1\le i \le n-1} \norm{\omega_{n}'+a'_i-\omega_i'},
\end{gather*}
where $N$ stands for the Clarke normal cone ${(N=N^C)}$;

\item
if $X$ is Asplund,
then, for any $\tau\in]0,1[$,
there exist points
$\omega'_i\in\Omega_i\cap{B}_\la(\omega_i)$ ${(i=1,\ldots,n-1)}$ and $\omega'_n\in\Omega_n\cap \B_\eta(\omega_n)$ satisfying
\begin{multline*}
0<\max_{1\le i \le n-1} \norm{\omega_{n}'+a'_i-\omega_i'} \\ <\varphi\iv(\varphi(d(\Omega_1-\omega_i-a_1,\ldots, \Omega_{n-1}-\omega_{n-1}-a_{n-1},\Omega_n-\omega_n))+\eps),
\end{multline*}
and vectors
$x_i^*\in X^*$ $(i=1,\ldots,n)$ satisfying conditions \eqref{T17-1}, \eqref{P21-4}, with $N$ standing for the \Fr\ normal cone ${(N=N^F)}$, and
\begin{align*}
\sum_{i=1}^{n-1}\ang{x_i^*,\omega_{n}'+a'_i-\omega_i'}>\tau \max_{1\le i \le n-1}\norm{\omega_{n}'+a'_i-\omega_i'}.
\end{align*}
\end{enumerate}
\end{proposition}

\begin{remark}
\begin{enumerate}
\item
In the particular case when all the points $\omega_i\in\Omega_i$ ${(i=1,\ldots,n)}$ coincide, i.e. $\omega_1=\ldots=\omega_n=:\bx\in\cap_{i=1}^n\Omega_i$, Proposition~\ref{P21} reduces to Theorem~\ref{T17}.
\item
Proposition~\ref{P21} does not assume the sets $\Omega_1,\ldots,\Omega_n$ to have a common point and, given some individual points
$\omega_i\in\Omega_i$ ${(i=1,\ldots,n)}$, establishes the existence of another collection of points
$\omega'_i\in\Omega_i$ ${(i=1,\ldots,n)}$ with certain properties, each point in a \nbh\ of the corresponding given one.
If the sets do have a common point $\bx\in\cap_{i=1}^n\Omega_i$, then, with $\xi:=\max_{1\le i \le n} \norm{\omega_i-\bx}$, the estimates in Proposition~\ref{P21} yield $\omega'_i\in \B_{\la+\xi}(\bx)$ ${(i=1,\ldots,n-1)}$ and $\omega'_n\in \B_{\eta+\xi}(\bx)$.
In this form, Proposition~\ref{P21} can be considered as an 
extension of \cite[Theorem~3.1]{KruLop12.1}, which served as the main tool when extending the extremal principle to infinite collections of sets.
Note that in the linear case, the conclusions of Proposition~\ref{P21} admit significant simplifications
which make the reduction of Proposition~\ref{P21} to (the improved version of) \cite[Theorem~3.1]{KruLop12.1} straightforward; cf. \cite[Proposition~6.1]{BuiKru19}.
\end{enumerate}
\end{remark}

Another consequence of Theorem~\ref{T17}
provides dual necessary conditions for a collection of sets with empty intersection.
It generalizes and improves \cite[Theorem 6.2]{BuiKru19}.

\begin{proposition}\label{ZhNg}
Let $\Omega_1,\ldots,\Omega_n$ be closed subsets of a Banach space $X$, $\omega_i\in\Omega_i$ ${(i=1,\ldots,n)}$, $\eps>0$ and $\varphi\in\mathcal{C}^1$.
Suppose that $\cap_{i=1}^n\Omega_i=\emptyset$ and
\begin{gather}\label{ZN-1}
\varphi\Big(\max_{1\le i\le n-1} \norm{\omega_n-\omega_i}\Big) <\varphi(d(\Omega_1,\ldots,\Omega_n))+\eps.
\end{gather}
Then, for any $\la>0$ and $\eta>0$,
\begin{enumerate}
\item
there exist points
$\omega'_i\in\Omega_i\cap{B}_\la(\omega_i)$ $(i=1,\ldots,n-1)$ and $\omega'_n\in\Omega_n\cap \B_\eta(\omega_n)$ satisfying
condition \eqref{P16-3}, and vectors
$x_i^*\in X^*$ $(i=1,\ldots,n)$ satisfying conditions \eqref{T17-1} and
\begin{gather}
\label{ZN-2}
\varphi'\Big(\max_{1\le i \le n-1} \norm{\omega'_n-\omega'_i}\Big) \Big({\la}\sum_{i=1}^{n-1} d\left(x_i^*, {N}_{\Omega_i}(\omega_i)\right) + {\eta}d\left(x_n^*,{N}_{\Omega_n}(\omega_n)\right)\Big) <\eps,
\\
\notag
\sum_{i=1}^{n-1}\ang{x_i^*,\omega'_n-\omega'_i}=\max_{1\le i \le n-1}\norm{\omega'_n-\omega'_i},
\end{gather}
where $N$ stands for the Clarke normal cone ${(N=N^C)}$;
\item
if $X$ is Asplund,
then, for any $\tau\in]0,1[$,
there exist points
$\omega'_i\in\Omega_i\cap{B}_\la(\omega_i)$ ${(i=1,\ldots,n-1)}$ and $\omega'_n\in\Omega_n\cap \B_\eta(\omega_n)$ satisfying
\begin{gather}\label{ZN-4}
0<\max_{1\le{i}\le{n}-1} \|\omega'_n-\omega'_{i}\| <\varphi\iv(\varphi(d(\Omega_1,\ldots,\Omega_n))+\eps),
\end{gather}
and vectors
$x_i^*\in X^*$ $(i=1,\ldots,n)$ satisfying conditions \eqref{T17-1}, \eqref{ZN-2}, with $N$ standing for the \Fr\ normal cone ${(N=N^F)}$, and
\begin{gather}
\notag
\sum_{i=1}^{n-1}\langle x^*_i, \omega'_n-\omega'_i\rangle >\tau\max_{1\le i \le n-1}\norm{\omega'_n-\omega'_i}.
\end{gather}
\end{enumerate}
\end{proposition}

\begin{proof}
It is sufficient to notice that the sets $\Omega_i':=\Omega_i-\omega_i$ $(i=1,\ldots,n)$ and vectors $a_i:=\omega_n-\omega_i$ $(i=1,\ldots,n-1)$ satisfy $0\in\cap_{i=1}^n\Omega_i'$ and \eqref{23}, and apply Theorem~\ref{T17}.
\end{proof}

\begin{remark}
Proposition~\ref{ZhNg} can be considered as an
extension and improvement of the two \emph{unified separation theorems} due to Zheng and Ng \cite[Theorems~3.1 and 3.4]{ZheNg11}, which correspond to $\varphi$ being the identity function, and $\la=\eta$; cf. \cite[Theorem~6.2]{BuiKru19}.
Note that conditions \eqref{P16-3} and \eqref{ZN-4} in Proposition~\ref{ZhNg}, having no analogues in \cite{ZheNg11}, provide additional restrictions on the choice of the points $\omega'_i\in\Omega_i$ ${(i=1,\ldots,n)}$.
For instance, if $\max_{1\le i\le n-1} \norm{\omega_n-\omega_i}=d(\Omega_1,\ldots,\Omega_n)$, i.e. the distance-like quantity \eqref{d} is attained at the points $\omega_i\in\Omega_i$ ${(i=1,\ldots,n)}$, then, thanks to \eqref{P16-3}, the points $\omega'_i\in\Omega_i$ ${(i=1,\ldots,n)}$ in part (i) of Proposition~\ref{ZhNg} also possess this property.

In \cite{ZheNg11}, instead of the distance-like quantity \eqref{d} in condition \eqref{ZN-1}, a slightly more general $p$-weighted nonintersect index was used with the corresponding $q$-weighted sums replacing the usual ones in \eqref{T17-1} and \eqref{ZN-2}.
This corresponds to considering $l_p$ norms on product spaces and the corresponding $l_q$ dual norms; cf. Remark~\ref{R5}(ii).
\end{remark}

To complete the section, we formulate two corollaries for the H\"older case, which correspond to setting $\varphi(t):=\al\iv t^q$ in Theorems~\ref{T17} and \ref{T19}, respectively.

\begin{corollary}\label{C4.5}
Let $\Omega_1,\ldots,\Omega_n$ be closed subsets of a Banach space $X$,
$\bx\in\cap_{i=1}^n\Omega_i$, ${a_i\in{X}}$ $(i=1,\ldots,n-1)$, $\eps>0$ and $q>0$.
Suppose that condition \eqref{P10-1} is satisfied, and
\begin{gather}\notag
\max_{1\le i\le n-1}\norm{a_i}^q< d^q(\Omega_1-a_1,\ldots,\Omega_{n-1}-a_{n-1}, \Omega_n)+\eps.
\end{gather}
Then, for any $\la>0$ and $\eta>0$,
\begin{enumerate}
\item
there exist points
$\omega_i\in\Omega_i\cap{B}_\la(\bx)$ $(i=1,\ldots,n-1)$ and $\omega_n\in\Omega_n\cap \B_\eta(\bx)$ satisfying condition \eqref{T12-3}, and vectors
$x_i^*\in X^*$ $(i=1,\ldots,n)$ satisfying conditions \eqref{T17-1}, \eqref{T17-3} and
\begin{gather}
\label{C4.5-2}
\hspace{-3mm} q\max_{1\le{i}\le{n}-1} \|\omega_n+a_i-\omega_{i}\|^{q-1} \Big({\la}\sum_{i=1}^{n-1} d\left(x_i^*, {N}_{\Omega_i}(\omega_i)\right) + {\eta}d\left(x_n^*,{N}_{\Omega_n}(\omega_n)\right)\Big) <\eps,
\end{gather}
where $N$ stands for the Clarke normal cone ${(N=N^C)}$;

\item
if $X$ is Asplund,
then, for any $\tau\in]0,1[$,
there exist points
$\omega_i\in\Omega_i\cap{B}_\la(\bx)$ ${(i=1,\ldots,n-1)}$ and $\omega_n\in\Omega_n\cap \B_\eta(\bx)$ satisfying
\begin{gather}
\notag
0<\max_{1\le{i}\le{n}-1} \|\omega_n+a_i-\omega_{i}\|^q <d^q(\Omega_1-a_1,\ldots,\Omega_{n-1}-a_{n-1}, \Omega_n)+\eps,
\end{gather}
and vectors
$x_i^*\in X^*$ $(i=1,\ldots,n)$ satisfying conditions \eqref{T17-1}, \eqref{T17-5} and \eqref{C4.5-2}, with $N$ standing for the \Fr\ normal cone ${(N=N^F)}$.
\end{enumerate}
\end{corollary}

\begin{corollary}
Let $\Omega_1,\ldots,\Omega_n$ be closed subsets of a Banach space $X$,
$\bx\in\cap_{i=1}^n\Omega_i$, ${a_i\in{X}}$ $(i=1,\ldots,n)$, $\rho\in]0,+\infty]$, $\eps>0$ and $q>0$.
Suppose that condition \eqref{D1-3} is satisfied and \begin{gather}\notag
\max_{1\le i\le n}\norm{a_i}^q <d^q(\Omega_1-a_1,\ldots,\Omega_{n}-a_{n}, {B}_\eta(\bx))+\eps.
\end{gather}
Then, for any $\la>0$ and $\eta\in]0,\rho[$,
\begin{enumerate}
\item
there exist points
$\omega_i\in\Omega_i\cap{B}_\la(\bx)$ $(i=1,\ldots,n)$ and $x\in \B_\eta(\bx)$ satisfying condition \eqref{T14-3}, and vectors
$x_i^*\in X^*$ $(i=1,\ldots,n)$ satisfying $\sum_{i=1}^{n}\norm{x_i^*}=1$, condition \eqref{T19-3} and
\begin{gather}
\label{C4.6-2}
q\max_{1\le{i}\le{n}} \|x+a_i-\omega_{i}\|^{q-1} \left({\la}\sum_{i=1}^{n} d\left(x_i^*, {N}_{\Omega_i}(\omega_i)\right) + {\eta}\norm{\sum_{i=1}^n x_i^*}\right)<\eps,
\end{gather}
where $N$ stands for the Clarke normal cone ${(N=N^C)}$;
\item
if $X$ is Asplund,
then, for any $\tau\in]0,1[$,
there exist points
$\omega_i\in\Omega_i\cap{B}_\la(\bx)$ $(i=1,\ldots,n)$ and $x\in \B_\eta(\bx)$ satisfying
\begin{gather}
\notag
0<\max_{1\le{i}\le{n}} \|x+a_i-\omega_{i}\|^q <d^q(\Omega_1-a_1,\ldots,\Omega_{n}-a_{n}, {B}_\eta(\bx))+\eps,
\end{gather}
and vectors
$x_i^*\in X^*$ $(i=1,\ldots,n)$ satisfying $\sum_{i=1}^{n}\norm{x_i^*}=1$, and conditions \eqref{T19-5} and \eqref{C4.6-2} with $N$ standing for the \Fr\ normal cone ${(N=N^F)}$.
\end{enumerate}
\end{corollary}

\section{Applications to Convergence Analysis of Alternating Projections}
\label{AP}

In this section, $X$ is a real Hilbert (not necessary finite dimensional) space with inner product $\ang{\cdot,\cdot}$.
For simplicity, we consider here the H\"older case.
In the setting of two sets,
Corollary~\ref{C4.5} yields the following statement.

\begin{proposition}\label{P5.1}
Let $A$ and $B$ be closed subsets of $X$, $\bx \in A\cap B$, $u\in X$, $\eps >0$ and $q>0$.
Suppose that $(A-u)\cap B=\emptyset$ and
$\norm{u}^q<d^q(A-u,B)+\eps$.
Then, for any $\la>0$ and $\eta>0$, there exist points $a\in A\cap \B_\la(\bx)$ and $b\in B\cap\B_\eta(\bx)$ such that
\begin{gather}
\label{P5.1.2}
\norm{b-a+u}>0,\quad \norm{b-a+u}^q<d^q(A-u,B)+\eps,\\
\label{P5.1.3}
q\norm{b-a+u}^{q-2}\Big(\la d(b-a+u,N_A^F(a)) + \eta d(a-b-u,N_B^F(b))\Big)<{\eps}.
\end{gather}
\end{proposition}

\begin{proof}
Let numbers $\la>0$ and $\eta>0$ be given.
Choose numbers $\eps'\in]0,\eps[$ and $\tau \in ]0,1[$ such that $\norm{u}^q<d^q(A-u,B)+\eps'$ and
\begin{equation*}
1-\tau <\frac{1}{2}\left(\frac{\eps - \eps'}{q(\la+\eta)(d^q(A-u,B)+\eps)^{1-1/q}}\right)^2.
\end{equation*}
By Corollary~\ref{C4.5}, there exist points $a\in A\cap \B_\la(\bx)$ and $b\in B\cap \B_\eta(\bx)$ such that $\norm{b-a+u}>0$ and $\norm{b-a+u}^q<d^q(A-u,B)+\eps'$, and a vector $x^*\in X$ such that $\norm{x^*}=1$ and
\begin{gather*}
q\norm{b-a+u}^{q-1}\big(\la d\left(x^*,N_A(a)\right) + \eta d\left(-x^*, N_B(b)\right)\big)<\eps',
\\
\ang{x^*,b-a+u}>\tau\norm{b-a+u}.
\end{gather*}
Thus, conditions \eqref{P5.1.2} are satisfied, and
\begin{align*}
\norm{\frac{b-a+u}{\norm{b-a+u}}-x^*} &=\sqrt{2-2\ang{x^*,\frac{b-a+u}{\norm{b-a+u}}}} <\sqrt{2(1-\tau)}
\\
&<\frac{\eps - \eps'} {q(\la+\eta)(d^q(A-u,B)+\eps)^{1-1/q}}
<\frac{\eps - \eps'} {q(\la+\eta)\norm{b-a+u}^{q-1}}.
\end{align*}
Hence,
\begin{align*}
&\la d\left(\frac{b-a+u}{\norm{b-a+u}},N_A(a)\right) +\eta d\left(\frac{a-b-u}{\norm{b-a+u}} , N_B(b)\right)
\\
&\hspace{2cm}\le \la d\left(x^*,N_A(a)\right) + \eta d\left(-x^*, N_B(b)\right)+ (\la+\eta)\norm{\frac{b-a+u}{\norm{b-a+u}}-x^*} \\
&\hspace{2cm}< \frac{\eps'}{q\norm{b-a+u}^{q-1}}+ {\frac{\eps - \eps'}{q\norm{b-a+u}^{q-1}}}= \frac{\eps}{q\norm{b-a+u}^{q-1}}.
\end{align*}
This proves \eqref{P5.1.3}.
\end{proof}

The next corollary presents a particular case of Proposition~\ref{P5.1} with $B=\{\bx\}$ and $u=b-\bx$ for some $b\notin A$.

\begin{corollary}\label{C5.2}
Let $A$ be a closed subset of $X$, $\bx \in A$, $\eps >0$ and $q>0$.
Suppose that $b\notin A$ and
$\norm{b-\bx}^q<d^q(b,A)+\eps$.
Then, for any $\la>0$, there exists a point $a\in A\cap \B_\la(\bx)$ such that $\norm{b-a}^q<d^q(b,A)+\eps$ and
$q\la\norm{b-a}^{q-2} d(b-a,N_A^F(a))<{\eps}$.
\end{corollary}

Corollary~\ref{C5.2} can be easily `reversed' to produce a sufficient condition for the `distance decrease' (cf. \cite[Theorem~5.2]{DruIofLew15}).

\begin{corollary}\label{C5.3}
Let $A$ be a closed subset of $X$, $\bx \in A$, $\de>0$, $\la>0$ and $q>0$.
Suppose that $b\notin A$ and
$\norm{b-a}^{q-2} d(b-a,N_A^F(a))\ge\de$
for all $a\in A\cap \B_\la(\bx)$ with $\norm{b-a}^q<d^q(b,A)+q\la\de$.
Then
$d^q(b,A)\le\norm{b-\bx}^q-q\la\de$.
\end{corollary}

Now we apply Corollary~\ref{C5.3} to {\it alternating projections}, i.e. a sequence $(x_n)$ satisfying with some initial point $x_0\in X$ the following conditions:
\begin{equation}\label{ap}
x_{2n-1}\in P_B(x_{2n-2})\AND
x_{2n}\in P_A(x_{2n-1}),\quad n=1,2,\ldots,
\end{equation}
where $P_A$ and $P_B$ stand for projection operators on the respective sets, e.g., $P_A(x):=\{a\in A:\|x-a\|=d(x,A)\}$.
Note that, if $a\in P_A(x)$, then $x-a\in N_A^F(a)$.
Note also that, for a sequence $(x_n)$ satisfying \eqref{ap}, it holds $\norm{x_{n+1}-x_{n}}\le\norm{x_{n}-x_{n-1}}$, $n=1,2,\ldots$, although the last estimate does not entail in general that $\norm{x_{n+1}-x_{n}}\downarrow d(A,B)$.

Alternating projections (\emph{von Neumann algorithm}) are traditionally used for solving \emph{feasibility problems}, i.e. finding a point in the intersection $A\cap B$ (cf. \cite{LewLukMal09,DruIofLew15,KruLukTha18}).
If $A\cap B=\es$, the sequence obviously cannot converge.
However, alternating projections can still be applied for locating a pair of points $a\in A$ and $b\in B$ such that $\norm{a-b}=d(A,B)$.
If $x_{2n}\to a\in A$ and $x_{2n-1}\to b\in B$ as $n\to+\infty$, with some abuse of terminology, we will say that alternating projections converge to $(a,b)$.

\begin{proposition}
\label{C5.4}
Let $A$ and $B$ be closed subsets of $X$, $\de>0$ and $q>0$.
Suppose that
\begin{equation}
\label{C5.4-1}
\max\left\{\norm{b-a}^{q-2}d\left(b-a,N_A^F(a) \right),\norm{b-a}^{-1}d\left(a-b,N_B^F(b)\right) \right\}\ge\de
\end{equation}
for all $a\in A$ and $b\in B$ such that 
$d(b,A)>d(A,B)$.
Then a sequence $(x_n)$ of alternating projections satisfies
\begin{equation}
\label{C5.4-2}
\norm{x_{2n}-x_{2n-1}}^q\le \norm{x_{2n-1}-x_{2n-2}}^q-q\de^2 \norm{x_{2n-1}-x_{2n-2}},
\end{equation}
as long as $d(x_{2n-1},A)>d(A,B)$.
Otherwise, $\norm{x_{2n}-x_{2n-1}}=d(A,B)$, and the distance $d(A,B)$ is attained at $x_{2n-1}\in B$ and $x_{2n}\in A$.

If  $\de<1$, then condition $a\in A$ in the assumption part can be replaced with $a\in A\setminus B$.
\end{proposition}
\if{
\AK{18/04/20.
Can conditions \eqref{C5.4-1} and/or \eqref{C5.4-2} hold with $q>1$?
How large can $\de$ be?

Can the convergence of (the norms in) \eqref{C5.4-2} be shown when $q\ne1$?
Is it obvious that a sequence satisfying \eqref{C5.4-2} is convergent?
If $q=1$, does the convergence rate coincide with the known one?}

\AK{20/04/20.
I am unable to prove Proposition~\ref{C5.4} if $d(b,A)>d(A,B)$ is replaced with
$\min\{d(a,B),d(b,A)\}>d(A,B)$.
It would be good to improve condition \eqref{C5.4-2} by removing the difference between odd and even iterations.
}

\AK{29/04/20.
Now condition \eqref{C5.4-1} is not symmetric, with the nonlinearity only present in the first term.
Unfortunately I have not been able to reproduce your estimates.
}
}\fi

\begin{proof}
Let $x_{2n-2}\in A$ and $x_{2n-1}\in P_B(x_{2n-2})$ for some $n\in\N$.
Suppose $d(x_{2n-1},A)>d(A,B)$.
Set $\la:=\de \norm{x_{2n-2}-x_{2n-1}}$.
Take any point $a \in A\cap \B_{\la}(x_{2n-2})$.
Since $x_{2n-1}\notin A$, we have $a\ne x_{2n-1}$ and $x_{2n-2}\ne x_{2n-1}$.
Moreover, if $\de<1$, then,
$\la< \norm{x_{2n-2}-x_{2n-1}}= d(x_{2n-2},B)$, and consequently, $a\notin B$.
A simple geometric argument involving the corresponding unit vectors shows that
\begin{align*}
d\left(\frac{a-x_{2n-1}}{\norm{a- x_{2n-1}}},\R_+(x_{2n-2} -  x_{2n-1})\right)=d\left(\frac{x_{2n-2} -  x_{2n-1}}{\norm{x_{2n-2} -  x_{2n-1}}},\R_+(a- x_{2n-1})\right).
\end{align*}
As a consequence, we have
\begin{align*}
d\left(a-x_{2n-1},N_B^F(x_{2n-1})\right)
&\le d\left(a-x_{2n-1},\R_+(x_{2n-2} -  x_{2n-1})\right) \\
&= \frac{\norm{a- x_{2n-1}}}{\norm{x_{2n-2} -  x_{2n-1}}}
d\left({x_{2n-2} -  x_{2n-1}},\R_+(a- x_{2n-1})\right) \\
&\le\frac{\norm{a- x_{2n-1}}\cdot\norm{a- x_{2n-2}}}{\norm{x_{2n-2} -  x_{2n-1}}}\\
&<\frac{\la\norm{a- x_{2n-1}}}{\norm{x_{2n-2} -  x_{2n-1}}}=\de\norm{a- x_{2n-1}}.
\end{align*}
In view of \eqref{C5.4-1}, we have
$
{\norm{x_{2n-1}-a}^{q-2}}d\left({x_{2n-1}-a}, N_A^F(a)\right) \ge\de.
$
By
{Corollary~\ref{C5.3}},
inequality \eqref{C5.4-2} holds true.
If $d(x_{2n-1},A)=d(A,B)$, then, since $x_{2n}\in P_A(x_{2n-1})$, we have $\norm{x_{2n}-x_{2n-1}}=d(A,B)$, i.e. the distance $d(A,B)$ is attained at $x_{2n-1}\in B$ and $x_{2n}\in A$.
\end{proof}

\begin{remark}
\begin{enumerate}
\item
Under the assumptions in Proposition~\ref{C5.4}, if $d(x_{2n-1},A)>d(A,B)$ for all $n\in\N$, then alternating projections make an infinite sequence satisfying \eqref{C5.4-2}.
As a consequence, the sequence $(\norm{x_{2n}-x_{2n-1}})$ is strictly decreasing while $\norm{x_{2n}-x_{2n-1}}>d(A,B)$ for all ${n\in\N}$.
If $d(x_{2n-1},A)=d(A,B)$ for some $n\in\N$, then after $2n$ projections we arrive at a pair of points $x_{2n-1}\in B$ and $x_{2n}\in A$ such that $\norm{x_{2n}-x_{2n-1}}=d(A,B)$.
In particular, if $A\cap B\ne\es$, then $x_{2n}=x_{2n-1}\in A\cap B$.

\item
Condition \eqref{C5.4-1} implies $\de\le\max\{\norm{b-a}^{q-1},1\}$, and, in particular, $\de\le1$ when either $q=1$, or $q>1$ and $\norm{b-a}\le1$.
\end{enumerate}
\end{remark}

\begin{corollary}
\label{C5.5}
Under the assumptions of Proposition~\ref{C5.4}, either $\norm{x_{n}-x_{n-1}}\to0$ as $n\to\infty$ or there is an $n\in\N$ such that $\norm{x_{n}-x_{n-1}}=d(A,B)$.
\end{corollary}

\begin{proof}
Suppose that $\norm{x_{n}-x_{n-1}}>d(A,B)$ for all $n\in\N$.
By Proposition~\ref{C5.4}, condition \eqref{C5.4-2} holds for all $n\in\N$.
If $\norm{x_{n}-x_{n-1}}>\al>0$ for all $n\in\N$, then, in view of \eqref{C5.4-2}, for all $n\in\N$, we have
$\norm{x_{2n+1}-x_{2n}}^q\le \norm{x_{2n}-x_{2n-1}}^q< \norm{x_{2n-1}-x_{2n-2}}^q-q\de^2\al$,
and consequently,
$\norm{x_{2n+1}-x_{2n}}^q< \norm{x_{1}-x_{0}}^q-nq\de^2\al$, which is not possible.
\end{proof}

The next statement provides the complete list of convergence estimates for alternating projections in the H\"older setting.

\begin{corollary}
\label{C5.6}
Suppose the assumptions of Proposition~\ref{C5.4} are satisfied.
\begin{enumerate}
\item
If either $d(A,B)>0$ or $q>1$, then there is an $n\in\N$ such that $\norm{x_{n}-x_{n-1}}=d(A,B)$.
\item
If $d(A,B)=0$, then $\norm{x_{n}-x_{n-1}}\to0$ as $n\to\infty$.
Moreover,
\begin{enumerate}
\item if
$q>1$, then there is an $n\in\N$ such that $x_{n}\in A\cap B$;
\item if
$q=1$, then for all $n\ge 1$, we have
\begin{gather}\label{C5.6-1}
	\norm{x_{2n}-x_{2n-1}}\le (1-\de^2)\norm{x_{2n-1}-x_{2n-2}};
\end{gather}
if $\dim X<\infty$, then $A\cap B\ne\es$, and the sequence converges to a point in $A\cap B$.
\end{enumerate}
\end{enumerate}
\end{corollary}

\begin{proof}
\begin{enumerate}
\item
If $d(A,B)>0$, then $\norm{x_{n}-x_{n-1}}\ge d(A,B)>0$, and the conclusion follows immediately from Corollary~\ref{C5.5}.
Let $d(A,B)=0$ and $q>1$.
Suppose that ${\norm{x_{n}-x_{n-1}}>0}$ for all $n\in\N$.
By Proposition~\ref{C5.4} and Corollary~\ref{C5.5}, condition \eqref{C5.4-2} holds for all $n\in\N$, and $\norm{x_{n}-x_{n-1}}\to0$ as $n\to\infty$.
In view of \eqref{C5.4-2}, for all $n\in\N$, we have
\begin{equation*}
\norm{x_{2n}-x_{2n-1}}^q\le \norm{x_{2n-1}-x_{2n-2}}^q\left(1-q\de^2 \norm{x_{2n-1}-x_{2n-2}}^{1-q}\right),
\end{equation*}
which is not possible as the \RHS\ of the last inequality becomes negative when $n$ is large enough.
\item
Let $d(A,B)=0$.
By Corollary~\ref{C5.5}, $\norm{x_{n}-x_{n-1}}\to0$ as $n\to\infty$.
Assertion (a) is a consequence of (i).
When $q=1$, condition \eqref{C5.4-2} reduces to \eqref{C5.6-1}.
The latter condition implies that the sequence $(x_n)$ is bounded.
If $\dim X<\infty$, it converges to a point in $A\cap B$.
\end{enumerate}
\end{proof}

\begin{remark}
Condition \eqref{C5.6-1} recaptures the $R-$linear convergence estimate in \cite[Theorem~6.1]{DruIofLew15} and \cite[Theorem~9.28]{Iof17}.
Note that in \cite{DruIofLew15,Iof17} a local setting is considered and the $R-$linear convergence (when $q=1$) of alternating projections is established near a given point $\bx\in A\cap B$.
This setting is covered by Corollary~\ref{C5.6}(iii) if the sets $A$ and $B$ are replaced by their intersections with some fixed \nbh\ of $\bx$, and the initial point $x_0$ is chosen sufficiently close to $\bx$.
\if{\AK{29/04/20.
\cite[Theorem~6.1]{DruIofLew15} claims a slightly weaker estimate.
To be checked.}}\fi
\end{remark}

\begin{example}\label{E5.5}
Consider a sequence of alternating projections for the pair of closed sets
$A:=\{(u,v): v\le 0\}$ and $B:=\{(u,v): v\ge |u|+1\}$ in the Euclidean space $\R^2$; see Fig.~\ref{Fig1}.
\begin{figure}[!ht]
			\centering
			\includegraphics[height=5cm,width=5.2cm]{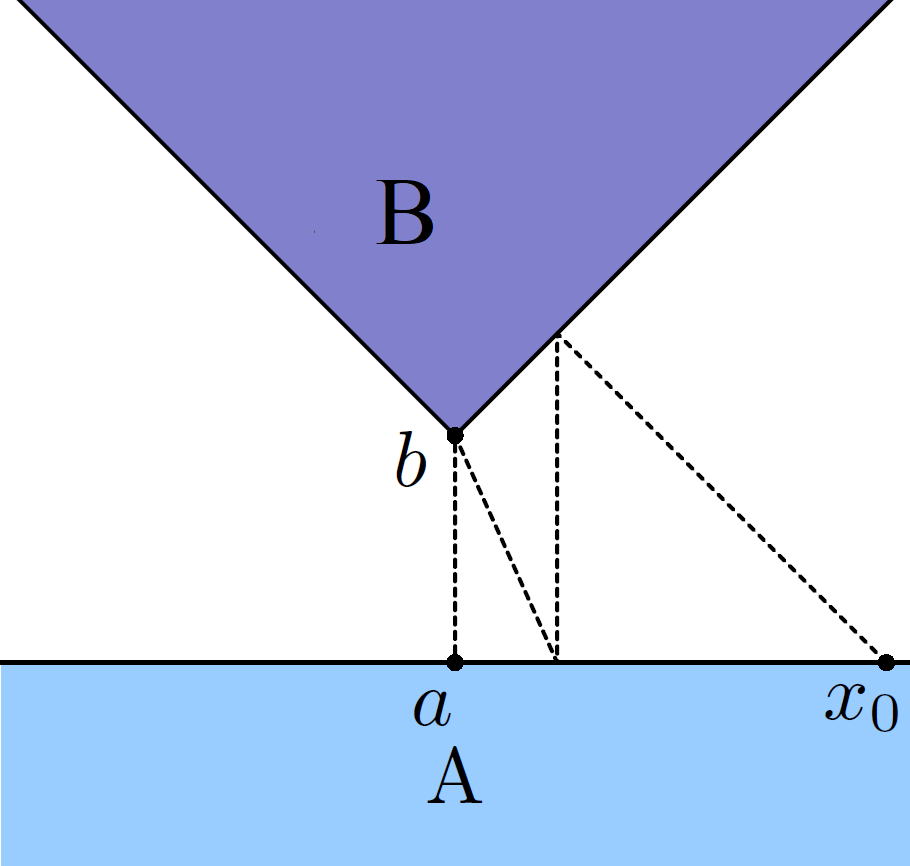}
			\caption{Example~\ref{E5.5}}
\label{Fig1}
	\end{figure}
Obviously, $d(A,B)=1$.
The setting verifies condition \eqref{C5.4-1} with $q=1$ except for the pair $(a,b)$ on which the distance $d(A,B)$ is attained.
In accordance with Corollary~\ref{C5.6}(i), the alternating projections arrive to this pair after a finite number of steps.
\end{example}

\section*{Acknowledgement(s)}

We would like to thank
PhD student Nguyen Duy Cuong from Federation University Australia for his careful reading of the manuscript.
We are also grateful to the referees for their constructive comments and criticism.

\section*{Disclosure statement}

No potential conflict of interest was reported by the authors.

\section*{Funding}

The research was supported by the Australian Research Council, project DP160100854.
The first author is supported by the Australian Research Council through the Centre for Transforming Maintenance through Data Science, project IC180100030.
The second author
benefited from the support of the European Union's Horizon 2020
research and innovation programme under the Marie Sk{\l}odowska--Curie
Grant Agreement No. 823731 CONMECH,
 and Conicyt REDES program 180032.

\addcontentsline{toc}{section}{References}
\def\cprime{$'$} \def\cftil#1{\ifmmode\setbox7\hbox{$\accent"5E#1$}\else
  \setbox7\hbox{\accent"5E#1}\penalty 10000\relax\fi\raise 1\ht7
  \hbox{\lower1.15ex\hbox to 1\wd7{\hss\accent"7E\hss}}\penalty 10000
  \hskip-1\wd7\penalty 10000\box7} \def\cprime{$'$} \def\cprime{$'$}
  \def\cprime{$'$} \def\cprime{$'$} \def\cprime{$'$}
  \def\Dbar{\leavevmode\lower.6ex\hbox to 0pt{\hskip-.23ex \accent"16\hss}D}
  \def\cfac#1{\ifmmode\setbox7\hbox{$\accent"5E#1$}\else
  \setbox7\hbox{\accent"5E#1}\penalty 10000\relax\fi\raise 1\ht7
  \hbox{\lower1.15ex\hbox to 1\wd7{\hss\accent"13\hss}}\penalty 10000
  \hskip-1\wd7\penalty 10000\box7} \def\cprime{$'$}


\begin{thebibliography}{10}
\providecommand{\url}[1]{\normalfont{#1}}
\providecommand{\urlprefix}{Available from: }

\bibitem{KruMor79}
Kruger~AY, Mordukhovich~BS. New necessary optimality conditions in problems of
  nondifferentiable programming. In: Numerical methods of nonlinear
  programming. Kharkov; 1979. p. 116--119. In Russian. \urlprefix
  \url{https://asterius.federation.edu.au/akruger/research/publications.html}.

\bibitem{KruMor80.2}
Kruger~AY, Mordukhovich~BS. Generalized normals and derivatives and necessary
  conditions for an extremum in problems of nondifferentiable programming. ii.
  Minsk; 1980. VINITI no.~494-80, 60 pp.; in Russian. \urlprefix
  \url{https://asterius.federation.edu.au/akruger/research/publications.html}.

\bibitem{KruMor80}
Kruger~AY, Mordukhovich~BS. Extremal points and the {E}uler equation in
  nonsmooth optimization problems. Dokl Akad Nauk BSSR.
  1980;\hspace{0pt}24(8):684--687. In Russian. \urlprefix
  \url{https://asterius.federation.edu.au/akruger/research/publications.html}.

\bibitem{DubMil65}
Dubovitskii~AY, Miljutin~AA. Extremum problems in the presence of restrictions.
  USSR Comput Maths Math Phys. 1965;\hspace{0pt}5:1--80.

\bibitem{Kru05}
Kruger~AY. Stationarity and regularity of set systems. Pac J Optim.
  2005;\hspace{0pt}1(1):101--126.

\bibitem{Kru06}
Kruger~AY. About regularity of collections of sets. Set-Valued Anal.
  2006;\hspace{0pt}14(2):187--206.
\urlprefix
  \url{https://doi.org/10.1007/s11228-006-0014-8}.

\bibitem{Kru09}
Kruger~AY. About stationarity and regularity in variational analysis. Taiwanese
  J Math. 2009;\hspace{0pt}13(6A):1737--1785.
\urlprefix
  \url{https://doi.org/10.1080/02331930902928674}.

\bibitem{DruIofLew15}
Drusvyatskiy~D, Ioffe~AD, Lewis~AS. Transversality and alternating projections
  for nonconvex sets. Found Comput Math. 2015;\hspace{0pt}15(6):1637--1651.
\urlprefix
  \url{http://dx.doi.org/10.1007/s10208-015-9279-3}.

\bibitem{Iof17}
Ioffe~AD. Variational analysis of regular mappings. theory and applications.
  Springer; 2017. Springer Monographs in Mathematics.
\urlprefix
  \url{http://dx.doi.org/10.1007/978-3-319-64277-2}.

\bibitem{Iof17.2}
Ioffe~AD. Transversality in variational analysis. J Optim Theory Appl.
  2017;\hspace{0pt}174(2):343--366.
\urlprefix
  \url{https://doi.org/10.1007/s10957-017-1130-3}.

\bibitem{KruLukTha17}
Kruger~AY, Luke~DR, Thao~NH. About subtransversality of collections of sets.
  Set-Valued Var Anal. 2017;\hspace{0pt}25(4):701--729.
\urlprefix
  \url{https://doi.org/10.1007/s11228-017-0436-5}.

\bibitem{Kru18}
Kruger~AY. About intrinsic transversality of pairs of sets. Set-Valued Var
  Anal. 2018;\hspace{0pt}26(1):111--142.
\urlprefix
  \url{https://doi.org/10.1007/s11228-017-0446-3}.

\bibitem{KruLukTha18}
Kruger~AY, Luke~DR, Thao~NH. Set regularities and feasibility problems. Math
  Program, Ser B. 2018;\hspace{0pt}168(1-2):279--311.
\urlprefix
  \url{https://doi.org/10.1007/s10107-017-1129-4}.

\bibitem{BivKraRib20}
Bivas~M, Krastanov~M, Ribarska~N. On tangential transversality. J Math Anal
  Appl. 2020;\hspace{0pt}481(1):article number 123445.
\urlprefix
  \url{https://doi.org/10.1016/j.jmaa.2019.123445}.

\bibitem{ThaBuiCuoVer20}
Thao~NH, Bui~TH, Cuong~ND, et~al. Some new characterizations of intrinsic
  transversality in {H}ilbert spaces. Set-Valued Var Anal.
  2020;\hspace{0pt}28(1):5--39.
\urlprefix
  \url{https://doi.org/10.1007/s11228-020-00531-7}.

\bibitem{Kru81.2}
Kruger~AY. Generalized differentials of nonsmooth functions. Minsk; 1981.
  VINITI no.~1332-81; 67 pp. In Russian. \urlprefix
  \url{https://asterius.federation.edu.au/akruger/research/publications.html}.

\bibitem{Kru85.1}
Kruger~AY. Generalized differentials of nonsmooth functions and necessary
  conditions for an extremum. Sibirsk Mat Zh. 1985;\hspace{0pt}26(3):78--90.
  (In Russian; English transl.: Siberian Math. J. 26 (1985), 370--379).

\bibitem{BorJof98}
Borwein~JM, Jofr{\'e}~A. A nonconvex separation property in {B}anach spaces.
  Math Methods Oper Res. 1998;\hspace{0pt}48(2):169--179.

\bibitem{Iof98}
Ioffe~AD. Fuzzy principles and characterization of trustworthiness. Set-Valued
  Anal. 1998;\hspace{0pt}6:265--276.

\bibitem{Kru03}
Kruger~AY. On {F}r\'{e}chet subdifferentials. J Math Sci (NY).
  2003;\hspace{0pt}116(3):3325--3358. 
\urlprefix
  \url{https://doi.org/10.1023/A:1023673105317}.

\bibitem{BorZhu05}
Borwein~JM, Zhu~QJ. Techniques of variational analysis. New York: Springer;
  2005.

\bibitem{Mor06.1}
Mordukhovich~BS. Variational analysis and generalized differentiation. {I}:
  {B}asic {T}heory. (Grundlehren der Mathematischen Wissenschaften [Fundamental
  Principles of Mathematical Sciences]; Vol. 330). Berlin: Springer; 2006.

\bibitem{Kru98}
Kruger~AY. About extremality of systems of sets. Dokl Nats Akad Nauk Belarusi.
  1998;\hspace{0pt}42(1):24--28. In Russian. \urlprefix
  \url{https://asterius.federation.edu.au/akruger/research/publications.html}.

\bibitem{Kru02}
Kruger~AY. Strict {$(\epsilon,\delta)$}-subdifferentials and extremality
  conditions. Optimization. 2002;\hspace{0pt}51(3):539--554.
\urlprefix
  \url{https://doi.org/10.1080/0233193021000004967}.

\bibitem{Kru04}
Kruger~AY. Weak stationarity: eliminating the gap between necessary and
  sufficient conditions. Optimization. 2004;\hspace{0pt}53(2):147--164.

\bibitem{BuiKru18}
Bui~HT, Kruger~AY. About extensions of the extremal principle. Vietnam J Math.
  2018;\hspace{0pt}46(2):215--242.
\urlprefix
  \url{https://doi.org/10.1007/s10013-018-0278-y}.

\bibitem{MorSha96}
Mordukhovich~BS, Shao~Y. Extremal characterizations of {A}splund spaces. Proc
  Amer Math Soc. 1996;\hspace{0pt}124(1):197--205.

\bibitem{Kru00}
Kruger~AY. Strict {$(\varepsilon,\delta)$}-semidifferentials and extremality of
  sets and functions. Dokl Nats Akad Nauk Belarusi.
  2000;\hspace{0pt}44(2):19--22. In Russian. \urlprefix
  \url{https://asterius.federation.edu.au/akruger/research/publications.html}.

\bibitem{BuiKru19}
Bui~HT, Kruger~AY. Extremality, stationarity and generalized separation of
  collections of sets. J Optim Theory Appl. 2019;\hspace{0pt}182(1):211--264.
\urlprefix
  \url{https://doi.org/10.1007/s10957-018-01458-8}.

\bibitem{BauBor93}
Bauschke~HH, Borwein~JM. On the convergence of von {N}eumann's alternating
  projection algorithm for two sets. Set-Valued Anal.
  1993;\hspace{0pt}1(2):185--212.

\bibitem{LewLukMal09}
Lewis~AS, Luke~DR, Malick~J. Local linear convergence for alternating and
  averaged nonconvex projections. Found Comput Math.
  2009;\hspace{0pt}9(4):485--513.
\urlprefix
  \url{https://doi.org/10.1007/s10208-008-9036-y}.

\bibitem{KruTha15}
Kruger~AY, Thao~NH. Quantitative characterizations of regularity properties of
  collections of sets. J Optim Theory Appl. 2015;\hspace{0pt}164(1):41--67.
\urlprefix
  \url{https://doi.org/10.1137/140991157}.

\bibitem{KruTha16}
Kruger~AY, Thao~NH. Regularity of collections of sets and convergence of
  inexact alternating projections. J Convex Anal.
  2016;\hspace{0pt}23(3):823--847.

\bibitem{BuiCuoKru20}
Bui~HT, Cuong~ND, Kruger~AY. Transversality of collections of sets: Geometric
  and metric characterizations. Vietnam J Math.
  2020;\hspace{0pt}48(2):277--297.
\urlprefix
  \url{https://doi.org/10.1007/s10013-020-00388-1}.

\bibitem{CuoKru}
Cuong~ND, Kruger~AY. Transversality properties: Primal sufficient conditions.
  Set-Valued Var Anal. 2020;\hspace{0pt}.
\urlprefix
  \url{https://doi.org/10.1007/s11228-020-00545-1}.

\bibitem{CuoKru20}
Cuong~ND, Kruger~AY. Nonlinear transversality of collections of sets: Dual
  space necessary characterizations. J Convex Anal.
  2020;\hspace{0pt}27(1):287--308.

\bibitem{CuoKru20.2}
Cuong~ND, Kruger~AY. Dual sufficient characterizations of transversality
  properties. Positivity. 2020;\hspace{0pt}24(5):1313--1359.
\urlprefix
  \url{https://doi.org/10.1007/s11117-019-00734-9}.

\bibitem{CuoKru6}
Cuong~ND, Kruger~AY. Primal necessary characterizations of transversality
  properties. Positivity. 2020;\hspace{0pt}.
\urlprefix
  \url{https://doi.org/10.1007/s11117-020-00775-5}.

\bibitem{KruLop12.1}
Kruger~AY, L\'{o}pez~MA. Stationarity and regularity of infinite collections of
  sets. J Optim Theory Appl. 2012;\hspace{0pt}154(2):339--369.
\urlprefix
  \url{https://doi.org/10.1007/s10957-012-0043-4}.

\bibitem{ZheNg06}
Zheng~XY, Ng~KF. The {L}agrange multiplier rule for multifunctions in {B}anach
  spaces. SIAM J Optim. 2006;\hspace{0pt}17(4):1154--1175.
\urlprefix
  \url{https://doi.org/10.1137/060651860}.

\bibitem{ZheNg11}
Zheng~XY, Ng~KF. A unified separation theorem for closed sets in a {B}anach
  space and optimality conditions for vector optimization. SIAM J Optim.
  2011;\hspace{0pt}21(3):886--911.
\urlprefix
  \url{https://doi.org/10.1137/100811155}.

\bibitem{ZheYanZou17}
Zheng~XY, Yang~Z, Zou~J. Exact separation theorem for closed sets in {A}splund
  spaces. Optimization. 2017;\hspace{0pt}66(7):1065--1077.
\urlprefix
  \url{https://doi.org/10.1080/02331934.2017.1316503}.

\bibitem{KruLop12.2}
Kruger~AY, L\'opez~MA. Stationarity and regularity of infinite collections of
  sets. {A}pplications to infinitely constrained optimization. J Optim Theory
  Appl. 2012;\hspace{0pt}155(2):390--416.
\urlprefix
  \url{https://doi.org/10.1007/s10957-012-0086-6}.

\bibitem{Aze03}
Az{\'e}~D. A survey on error bounds for lower semicontinuous functions. In:
  Proceedings of 2003 {MODE}-{SMAI} {C}onference; (ESAIM Proc.; Vol.~13). EDP
  Sci., Les Ulis; 2003. p. 1--17.

\bibitem{NgaThe08}
Ngai~HV, Th{\'e}ra~M. Error bounds in metric spaces and application to the
  perturbation stability of metric regularity. SIAM J Optim.
  2008;\hspace{0pt}19(1):1--20.
\urlprefix
  \url{https://doi.org/10.1137/060675721}.

\bibitem{Kru15}
Kruger~AY. Error bounds and metric subregularity. Optimization.
  2015;\hspace{0pt}64(1):49--79.
\urlprefix
  \url{https://doi.org/10.1080/02331934.2014.938074}.

\bibitem{AzeCor17}
Az\'{e}~D, Corvellec~JN. Nonlinear error bounds via a change of function. J
  Optim Theory Appl. 2017;\hspace{0pt}172(1):9--32.
\urlprefix
  \url{https://doi.org/10.1007/s10957-016-1001-3}.

\bibitem{ChuKruYao11}
Chuong~TD, Kruger~AY, Yao~JC. Calmness of efficient solution maps in parametric
  vector optimization. J Global Optim. 2011;\hspace{0pt}51(4):677--688.
\urlprefix
  \url{https://doi.org/10.1007/s10898-011-9651-z}.

\bibitem{RocWet98}
Rockafellar~RT, Wets~RJB. Variational analysis. Berlin: Springer; 1998.

\bibitem{DonRoc14}
Dontchev~AL, Rockafellar~RT. Implicit functions and solution mappings. a view
  from variational analysis. 2nd ed. New York: Springer; 2014. Springer Series
  in Operations Research and Financial Engineering.
\urlprefix
  \url{https://doi.org/10.1007/978-1-4939-1037-3}.

\bibitem{DegMarTos80}
De~Giorgi~E, Marino~A, Tosques~M. Evolution problerns in in metric spaces and
  steepest descent curves. Atti Accad Naz Lincei Rend Cl Sci Fis Mat Natur (8).
  1980;\hspace{0pt}68(3):180--187. In Italian. English translation: Ennio De
  Giorgi, Selected Papers, Springer, Berlin 2006, 527--533.

\bibitem{Cla83}
Clarke~FH. Optimization and nonsmooth analysis. New York: John Wiley \& Sons
  Inc.; 1983.

\bibitem{Roc79}
Rockafellar~RT. Directionally {L}ipschitzian functions and subdifferential
  calculus. Proc London Math Soc (3). 1979;\hspace{0pt}39(2):331--355.

\bibitem{Fab89}
Fabian~M. Subdifferentiability and trustworthiness in the light of a new
  variational principle of {B}orwein and {P}reiss. Acta Univ Carolinae.
  1989;\hspace{0pt}30:51--56.

\bibitem{Phe93}
Phelps~RR. Convex functions, monotone operators and differentiability. 2nd ed.
  (Lecture Notes in Mathematics; Vol. 1364). Springer-Verlag, Berlin; 1993.

\end{thebibliography}
\end{document}